\documentclass[graybox]{svmult}

\usepackage{mathptmx}       
\usepackage{helvet}         
\usepackage{courier}        
\usepackage{type1cm}        
\usepackage{makeidx}         
\usepackage{graphicx}        
\usepackage{multicol}        
\usepackage[bottom]{footmisc}
\usepackage{amssymb}
\usepackage{caption}
\usepackage{amsmath}

\def\f#1{{\mathcal F}_{#1}}
\def\m#1{{\mathcal M}_{#1}}
\def\a#1{{\mathcal A}_{#1}}

\def\eps{\varepsilon}

\def\un{1\!\!\!1}

\def\eps{\varepsilon}
\def\M{{\mathcal M}}

\def\e#1{\varepsilon_#1}

\def\R{{\mathbb R}}

\def\E{{\mathbb E}}

\newcommand{\Dim}{\mathop{\rm Dim}}
\newcommand{\diminf}{\mathop{\rm dim_*}}
\newcommand{\dimsup}{\mathop{\rm dim^*}}
\newcommand{\Diminf}{\mathop{\rm Dim_*}}
\newcommand{\Dimsup}{\mathop{\rm Dim^*}}

\makeindex             

\begin{document}

\title*{An introduction to Mandelbrot cascades}
\author{Yanick Heurteaux}
\institute{Yanick Heurteaux \at Clermont Universit\'e, Universit\'e Blaise Pascal, Laboratoire de Math\'ematiques, BP 10448, F-63000 CLERMONT-FERRAND -
CNRS, UMR 6620, Laboratoire de Math\'ematiques, F-63177 AUBIERE \email{Yanick.Heurteaux@math.univ-bpclermont.fr}
}
%
%
\maketitle

\abstract*{}

\abstract{In this course, we propose an elementary and self-contained introduction to canonical Mandelbrot random cascades. The multiplicative construction is explained and the necessary and sufficient condition of non-degeneracy is proved. Then, we discuss the problem of the existence of moments and the link with non-degeneracy. We also calculate the almost sure dimension of the measures. Finally, we give an outline on multifractal analysis of Mandelbrot cascades. This course was delivered in september 2013 during a meeting of the ``Multifractal Analysis GDR`` (GDR $\mbox{n}^{\small\mbox{o}}$ 3475 of the french CNRS).}
\section{Introduction}
\label{sec:0}
At the beginning of the seventies, Mandelbrot proposed a model of random measures based on an elementary multiplicative construction. This model, known as canonical Mandelbrot cascades,  was introduced to simulate the energy dissipation in intermittent turbulence (\cite{M74a}). In two notes (\cite{M74b} and \cite{M74c}) published in '74, Mandelbrot  described the fractal nature of the sets in which the energy is concentrated and proved or conjectured the main properties of this model. Two years later, in the fundamental paper \cite{KP76}, Kahane and Peyri\`ere proposed a complete proof of the results announced by Mandelbrot. In particular, the questions of non-degeneracy, existence of moments and dimension of the measures  were rigorously solved. 

Mandelbrot also observed that in a multiplicative cascade, the energy is distributed along a large deviations principle: this was the beginnings  of the multifractal analysis. 

Multifractal analysis developed a lot in the 80's. Frisch an Parisi observed that in the context of the fully developed  turbulence, the pointwise H\"older exponent of the dissipation of energy varies widely from point to point. They proposed in \cite{FP85} an heuristic argument, showing that the Hausdorff dimension of the level sets of a measure or a function can  be obtained  as the Legendre transform of a free energy function (which  will be called in this text  the structure function). This principle is known as {\it Multifractal Formalism}. Such a formalism was then rigorously proved by  Brown Michon and Peyri\`ere for the so called quasi-Bernoulli measures (\cite{BMP}). In particular, they highlighted the link between the multifractal formalism and the existence of auxiliary measures (known as Gibbs measures).

The problem of the multifractal analysis of Mandelbrot cascades appeared as a natural question at the end of the 80's.  Holley and Waymire were the first to obtain results in this direction. Under restrictive hypotheses, they proved in \cite{HW} that for any value of the H\"older exponent, the multifractal formalism is almost surely satisfied. The expected stronger result which says that, almost surely, for any value of the H\"older exponent, the multifractal formalism is satisfied was finally proved by Barral at the end of the $20^{\mbox{\scriptsize th}}$ century (\cite{B00}).

Let us finish this overview by saying that there exist now many generalizations of  the Mandelbrot cascades  (see for example \cite{BFP} for the description of the principal ones).

In the following pages, we want to relate the beginning of the story of canonical Mandelbrot cascades. As a preliminary, we explain the well known determinist case of binomial cascades. It allows us to describe the multiplicative principle, to introduce the most important notations and  definitions, and to show the way to calculate the dimension and to perform the multifractal analysis. Then, we introduce the canonical random Mandelbrot cascades (Theorem \ref{THEOexistence}), solve the problem of non-degeneracy  (Theorem \ref{THEOequi}) and its link with the existence of moments for the total mass of the cascade (Theorem \ref{THEOmoments}). In Section \ref{SECdim}, we prove that the Mandelbrot cascades are almost surely unidimensional and give the value of the dimension (Theorem \ref{THEOdim}). Finally, in a last section, we deal with the problem of multifractal analysis, and prove that for any value of the parameter $\beta$ the Hausdorff dimension of the level set of points with H\"older exponent $\beta$ is almost surely given by the multifractal formalism (Theorem \ref{THEOegal}). To obtain such a result, we use auxiliary cascades and we need to describe the simultaneous behavior of two cascades (Theorem \ref{THEOsimultaneous}) and to prove the existence of negative moments for the total mass (Proposition \ref{PROPnegative}).
\section{Binomial cascades}
\label{sec:1}
In order to understand the multiplicative construction principle, we begin with a very simple and classical example, known as Bernoulli product, which can be regarded as an introduction to the following. 

Let $\f n$ be the family of dyadic intervals of the
$n^{\mbox{\scriptsize th}}$ generation on $[0,1)$, $0<p<1$ and define the measure $m$ as 
follows. If $\varepsilon_1\cdots\varepsilon_n$ are integers
in $\{0,1\}$, and if
$$I_{\varepsilon_1\cdots\varepsilon_n}= \left[
\sum_{i=1}^n\frac{\varepsilon_i}{2^i},
\sum_{i=1}^n\frac{\varepsilon_i}{2^i}+\frac{1}{2^n}\right)\in\f
n$$ 
then
\begin{equation}\label{EQbernoulli}
m\left(I_{\varepsilon_1\cdots\varepsilon_n}\right)=p^{S_n}(1-p)^{n-S_n},
\quad\mbox{where}\quad S_n=\varepsilon_1+\cdots+\varepsilon_n\ .
\end{equation}
The measure $m$ is constructed using a multiplicative principle: if $I=I_{\varepsilon_1\cdots\varepsilon_n}\in \f n$ and in $I'=I_{\varepsilon_1\cdots\varepsilon_n0}$ and $I''=I_{\varepsilon_1\cdots\varepsilon_n1}$ are the two children of $I$ in $\f {n+1}$, then
$$m(I')=pm(I)\quad\mbox{and}\quad m(I'')=(1-p)m(I).$$
If $x\in[0,1)$, we can find $\e 1,\cdots,\e n,\cdots\in \{0,1\}$
uniquely determined  and such that for any $n\ge 1$, $x\in I_{\e 1\cdots\e n}$. We also denote $I_{\e 1\cdots\e n}=I_n(x)$ and we observe that
$$\frac{\log m(I_n(x))}{\log\vert I_n(x)\vert}=-\left(\frac{S_n}n\log_2 p+\left( 1-\frac{S_n}n\right)\log_2 (1-p)\right)$$
where $\vert I\vert$ is the length of the interval $I$. By the strong law of large numbers applied to the sequence $(\e n)$, we can then conclude that 
$$\lim_{n\to \infty}\frac{\log m(I_n(x))}{\log\vert I_n(x)\vert}=h(p)\quad dm-\mbox{almost surely}$$
where $h(p)=-(p\log_2 p+(1-p)\log_2(1-p))$.

Using Billingsley's Theorem (see for example \cite{Fal97}), it is then easy to conclude that $$\diminf(m)=\dimsup(m)=h(p)$$
where $\diminf(m)$ and $\dimsup(m)$ are the lower and the upper dimension
defined by
\begin{equation}\label{EQdim}
\left\{
\begin{aligned}
\diminf(m)&=\inf(\dim(E)\ ;\ m(E)>0)\\
\null\\
\dimsup(m)&=\inf(\dim(E)\ ;\ m([0,1]\setminus E)=0)
\end{aligned}
\right.
\end{equation}
It means that the measure $m$ is supported by a set of Hausdorff
dimension $h(p)$ and that every set of dimension less that $h(p)$
is negligible. We say that the measure $m$ is unidimensional with
dimension $h(p)$. 

If $\Dim(E)$ is the packing dimension of a set $E$ and if
\begin{equation}\label{EQDim}
\left\{
\begin{aligned}
\Diminf(m)&=\inf(\Dim(E)\ ;\ m(E)>0)\\
\null\\
\Dimsup(m)&=\inf(\Dim(E)\ ;\ m([0,1]\setminus E)=0)
\end{aligned}
\right.
\end{equation}
we can also conclude that 
$$\Diminf(m)=\Dimsup(m)=h(p).$$
\subsection{Multifractal analysis of binomial cascades}
Binomial cascades are also known to be multifractal measures and it is easy to compute their multifractal spectrum. Let 
$$E_\beta=\left\{x\ ;\ \lim_{n\to \infty}\frac{\log m(I_n(x))}{\log\vert I_n(x)\vert}=\beta\right\}$$
and recall that 
$$\frac{\log m(I_n(x))}{\log\vert I_n(x)\vert}=-\left(\frac{S_n}n\log_2 p+\left( 1-\frac{S_n}n\right)\log_2 (1-p)\right).$$
If $\beta\in[-\log_2 p,-\log_2(1-p)]$, we can find $\theta\in[0,1]$ such that 
$$\beta=-(\theta\log_2 p+(1-\theta)\log_2(1-p)).$$
It follows that 
$E_\beta=\left\{\frac{S_n}n\to\theta\right\}$ and we can conclude that
\begin{equation}\label{EQmultif}
\dim(E_\beta)=-(\theta\log_2\theta+(1-\theta)\log_2(1-\theta))=h(\theta) :=F(\beta)
\end{equation}
where
$F(\beta)=h\left(\frac{\beta+\log_2(1-p)}{\log_2(1-p)-\log_2 p}\right)$.
\begin{center}
\includegraphics[scale=0.3]{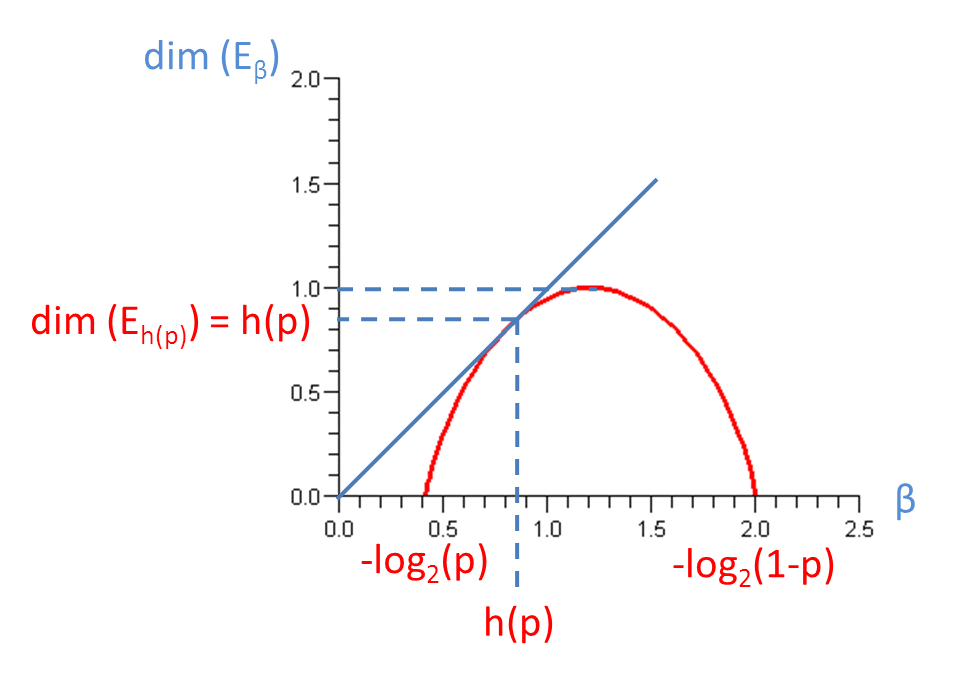}
\captionof{figure}{The spectrum of the measure $m$}
\end{center}
\subsection{Binomial cascades satisfy the multifractal formalism}
We can also rewrite formula (\ref{EQmultif}) in the following way. If $m_\theta$ be the binomial 
cascade with parameter $\theta$, the measure $m_\theta$ is supported by $E_\beta$ and we have
$$\dim(E_\beta)=\dim(m_\theta)=h(\theta).$$
Moreover, if $q\in\R$ is such that
$$\theta=\frac{p^q}{p^q+(1-p)^q},$$
and if $I\in\f n$, we have
\begin{eqnarray*}
 m_\theta(I)&=&\theta^{S_n}(1-\theta)^{n-S_n}\\
&=&\frac{p^{qS_n}(1-p)^{q(n-S_n)}}{\left(p^q+(1-p)^q\right)^n}\\
&=&m(I)^q\vert I\vert^{\tau(q)}
\end{eqnarray*}
where $\tau(q)=\log_2\left( p^q+(1-p)^q\right)$ is the structure function of the measure $m$ at state $q$.

Finally, if we observe that $\beta=-(\theta\log_2 p+(1-\theta)\log_2(1-p))=-\tau'(q)$, we can conclude that

\begin{eqnarray*}
 \dim(E_\beta)&=&-(\theta\log_2\theta+(1-\theta)\log_2(1-\theta))\\
&=&-q\tau'(q)+\tau(q)\\
&=&\tau^*(-\tau'(q))\\
&=&\tau^*(\beta)
\end{eqnarray*}
where $\tau^*(\beta)=\inf_t(t\beta+\tau(t))$ is the Legendre transform of $\tau$. 

We say that the measure $m$ satisfies the multifractal formalism and that $m_\theta$ is a Gibbs measure at state $q$. Such a construction of an auxiliary cascade will be used in Section \ref{SECmulti}.

\begin{remark}
The new measure $m_\theta$ is obtained from $m$ by changing the parameters $(p,1-p)$ in $\left(\frac{p^q}{p^q+(1-p)^q},\frac{(1-p)^q}{p^q+(1-p)^q}\right)$. The quantity $\frac1{p^q+(1-p)^q}$ is just the renormalization needed to ensure that the sum of the two parameters is equal to 1. A similar idea will be used to construct auxiliary Mandelbrot cascades (see the beginning of Section \ref{SECmulti}).
\end{remark}

\begin{remark}
If $m$ is a binomial cascade, we have 
$$\sum_{I\in{\mathcal F}_{n+1}}m(I)^q=\sum_{I\in\f n}p^qm(I)^q+(1-p)^qm(I)^q=\left(p^q+(1-p)^q\right)\sum_{I\in\f n}m(i)^q.$$
Finally,
$$\log_2\left( p^q+(1-p)^q\right)=\limsup_{n\to+\infty}\frac1n\log_2\left(\sum_{I\in\f n}m(I)^q\right)$$
which is the classical definition of the structure function $\tau$ (see Section \ref{SECdigression}).
\end{remark}
\subsection{Back to the existence of binomial cascades}
We want to finish this section with an elementary proposition which gives a rigorous proof of the existence of a measure $m$ satisfying (\ref{EQbernoulli}). Denote by $\lambda$ the Lebesgue measure on $[0,1)$ and let

$$m_n=f_nd\lambda\quad \mbox{where}\quad f_n=2^n\sum_{\e 1\cdots\e n}p^{S_n}(1-p)^{n-S_n}1\!\!\!1_{I_{\e 1\cdots\e n}}.$$
If $I=I_{\e 1\cdots\e j}\in\f j$, we have
$$m_j(I)=p^{S_j}(1-p)^{j-S_j}=m_{j+1}(I)=\cdots =m_{j+k}(I)=\cdots$$
and the sequence $(m_n(I))_{n\ge 1}$ is convergent. 
\begin{center}
\includegraphics[scale=0.3]{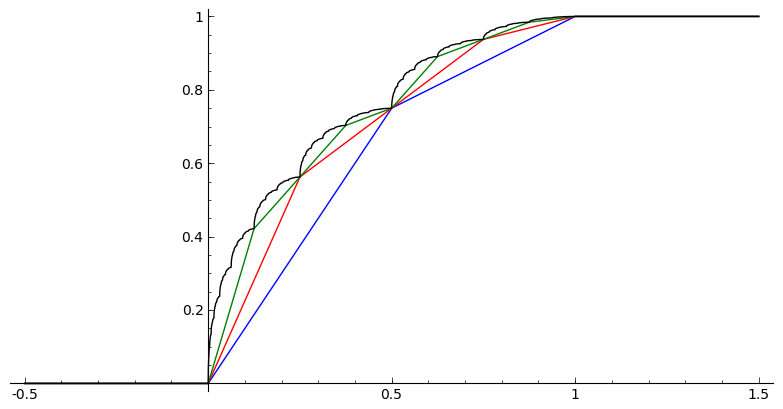}
\captionof{figure}{The repartition function of the measures $m_1$, $m_2$, $m_3$ and $m$}
\end{center}
We can then use the following elementary proposition.
\begin{proposition}\label{PROPweak}
 Let $(m_n)_{n\ge 1}$ be a sequence of finite Borel measures on $[0,1)$. Suppose that for any dyadic interval  $I\in\bigcup_{j\ge 0}\f j$, the sequence $(m_n(I))_{n\ge 1}$ is convergent. Then, the sequence $(m_n)_{n\ge 1}$ is weakly convergent to a finite Borel measure $m$.
\end{proposition}
\begin{remark}
 In Proposition \ref{PROPweak}, we can of course replace the family of dyadic intervals by the family of $\ell$-adic intervals ($\ell\ge 2$). Proposition \ref{PROPweak} will be used in Section \ref{sec:2} to prove the existence of Mandelbrot caseades.
\end{remark}
\begin{proof}[Proof of Proposition \ref{PROPweak}]
 Observe that if $f$ is a continuous function on $[0,1]$ and $\varepsilon>0$, we can find a function $\varphi$ which is a linear combinaison of functions $1\!\!\!1_I$ with $I\in\bigcup_{j\ge 0}\f j$ and such that $\Vert f-\varphi\Vert_\infty\le\varepsilon$.
 By the hypothesis, the sequence $\int \varphi(x)\,dm_n(x)$ is convergent and we have
\begin{eqnarray*}
\left\vert\int f\,dm_n-\int f\,dm_p\right\vert&\le&\left\vert\int \varphi\,dm_n-\int \varphi\,dm_p\right\vert+\Vert f-\varphi\Vert_\infty (m_n([0,1])+m_p([0,1]))\\
&\le& \left\vert\int \varphi\,dm_n-\int \varphi\,dm_p\right\vert+C\varepsilon .
\end{eqnarray*}
It follows that the sequence $\int f(x)\,dm_n(x)$ is convergent. The conclusion is then a consequence of the Banach-Steinhaus theorem and of the Riesz representation theorem.
\end{proof}

\section{Canonical Mandelbrot cascades : construction and non-degeneracy conditions}
\label{sec:2}
\subsection{Construction}
In all the sequel, $\ell\ge 2$ is an integer and $\f n$ is the set of $\ell$-adic intervals of the $n^{\mbox{\scriptsize th}}$ generation on $[0,1)$. We denote by $\m n$ the set of words of length $n$ written with the letters $0,\cdots,\ell-1$ and $\mathcal{M}=\bigcup_n\m n$. If $\eps=\e 1\cdots \e n\in\m n$, let
$$I_{\e 1\cdots\e n}=\left[\sum_{k=1}^n\frac{\e k}{\ell^k},\sum_{k=1}^n\frac{\e k}{\ell^k}+\frac1{\ell^n}\right)\in \f n.$$
Let $W$ be a non-negative random variable such that $E[W]=1$ and 
$(W_\eps)_{\eps\in\mathcal{M}}$ be a family of independent copies of $W$.

If $\lambda$ is the Lebesgue on $[0,1]$, we can define the sequence of random measures by 
$$m_n=f_n\lambda\quad\mbox{where}\quad f_n=\sum_{\e 1\cdots\e n\in\mathcal{M}}W_{\e 1}W_{\e 1\e 2}\cdots W_{\e 1\cdots\e n}1\!\!\!1_{I_{\e 1\cdots\e n}}.$$
The construction of the  measure $m_n$ uses a multiplicative principle and
$$m(I_{\e 1\cdots\e n})=\ell^{-n}W_{\e 1}W_{\e 1\e 2}\cdots W_{\e 1\cdots\e n}.$$
We have the following existence theorem :
\begin{theorem}[existence of $m$]\label{THEOexistence}
Almost surely, the sequence $(m_n)_{n\ge 1}$ is weakly convergent to a (random) measure $m$. The measure $m$ is called the Mandelbrot cascade associated to the weight $W$.
\end{theorem}
\begin{remark}
 The condition $E[W]=1$ is a natural condition. Indeed if 
$$Y_n:=m_n([0,1])=\int_0^1f_n(t)\,d\lambda(t)=\ell^{-n}\sum_{\e 1\cdots\e n\in\m n}W_{\e 1}W_{\e 1\e 2}\cdots
W_{\e 1\cdots\e n}$$
then,
\begin{eqnarray*}
E[Y_{n}]&=&\ell^{-n}\sum_{\e 1\cdots\eps_{n}\in\m n}E\left[W_{\e 1}W_{\e 1\e 2}\cdots W_{\e 1\cdots\eps_{n-1}}\right]E\left[{W_{\e 1\cdots\eps_{n}}}\right]\\
&=&E[Y_{n-1}]\times E[W]
\end{eqnarray*}
and the condition $E[W]=1$ ensures that the expectation of the total mass doesn't go to 0 or to $+\infty$.
\end{remark}

\begin{proof}
Let $\mathcal{A}_n$ be the $\sigma$-algebra generated by the $W_\eps$, $\eps\in\m 1\cup\cdots\cup\m n$. 
Define $Y_n:=m_n([0,1])$. An easy calculation says 

\begin{eqnarray*}
E[Y_{n+1}\vert \a n]&=&\ell^{-(n+1)}\sum_{\e 1\cdots\eps_{n+1}\in{\mathcal M}_{n+1}}
E[W_{\e 1}W_{\e 1\e 2}\cdots W_{\e 1\cdots\e n}W_{\e 1\cdots\eps_{n+1}}\vert\a n]\\
&=&\ell^{-(n+1)}\sum_{\e 1\cdots\eps_{n+1}\in{\mathcal M}_{n+1}}W_{\e 1}W_{\e 1\e 2}\cdots W_{\e 1\cdots\e n}E[W_{\e 1\cdots\eps_{n+1}}]\\
&=&Y_n
\end{eqnarray*} 
and the sequence $Y_n$ is a non negative martingale. So it is almost surely convergent.

More generally, if $I=I_{\alpha_1\cdots\alpha_k}\in\f k$,
$$
m_{k+n}(I)=
\ell^{-(k+n)}\sum_{\eps_{k+1}\cdots\eps_{k+n}\in\m n}W_{\alpha_1}\cdots W_{\alpha_1\cdots\alpha_k}W_{\alpha_1\cdots\alpha_k\eps_{k+1}}\cdots W_{\alpha_1\cdots\alpha_k\eps_{k+1}\cdots\eps_{k+n}}
$$
and a similar calculation says that $m_{k+n}(I)$ is a non-negative martingale. Finally, for any $I\in\bigcup_{k\ge 0}\f k$ the random quantity $m_n(I)$ is almost surely convergent. 

If we observe that the set $ \bigcup_{k\ge 0}\f k$ is countable, we can also say that almost surely, for any $I\in\bigcup_{k\ge 0}\f k$, $m_n(I)$ is convergent and the conclusion is a consequence of Proposition \ref{PROPweak}.
\end{proof}

\begin{center}
\includegraphics[scale=0.5]{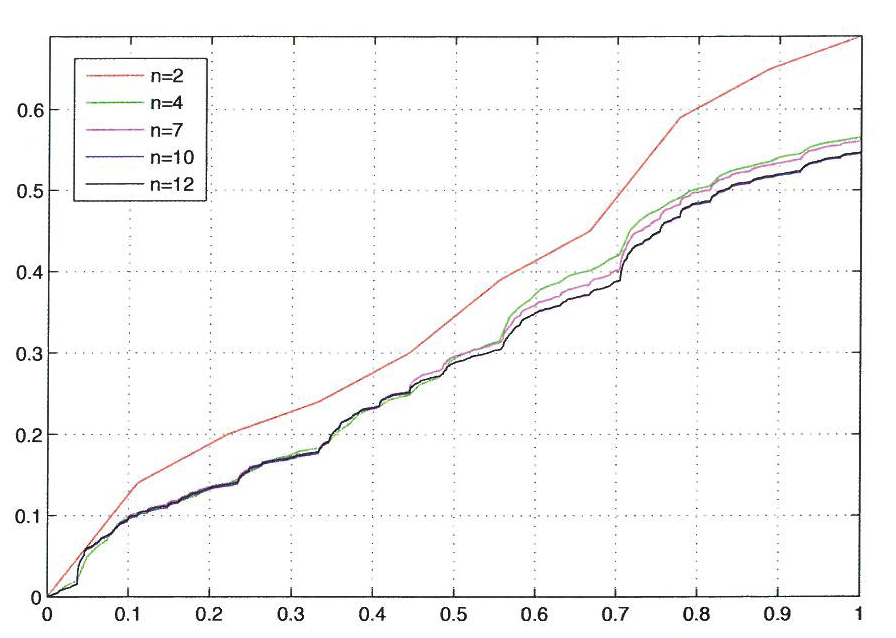}
\captionof{figure}{The repartition function of the random measures $m_n$ (from \cite{BPGaz})}
\end{center}
\subsection{Examples}
\subsubsection{Birth and death processes}
We suppose in this example that the random variable $W$ only takes the value 0 and another positive value. Let $p=1-P[W=0]$. To ensure that $E[W]=1$ we need to take $P\left[ W=\frac1p\right]=p$. When $m\not=0$, its support is a random Cantor set.
\subsubsection{Log-normal cascades}
This is the case where $W$ is a log-normal random variable, that is $W=e^X$ where $X$ follows a normal distribution with expectation $m$ and variance $\sigma^2$. An easy calculation says that 
\begin{eqnarray*}
E\left[e^X\right]&=&\int e^x\,e^{-(x-m)^2/2\sigma^2}\,\frac{dx}{\sigma\sqrt{2\pi}}\\
&=&\int e^{(m+\sigma u)}\,e^{-u^2/2}\,\frac{du}{\sqrt{2\pi}}\\
&=&\int e^{-(u-\sigma)^2/2}\,e^{m+\sigma^2/2}\,\frac{du}{\sqrt{2\pi}}\\
&=&e^{m+\sigma^2/2}
\end{eqnarray*}
In order to have $E[W]=1$ we need to choose $m=-\sigma^2/2$. In other words,
$$W=e^{\sigma N-\sigma^2/2}$$
where $N$ follows a standard normal distribution.
\subsection{The fundamental equations}
Define  $Y_n=m_n([0,1])$ as above. Then,
\begin{equation}\label{EQiter}
\begin{aligned}
Y_{n+1}&=\ell^{-(n+1)}\sum_{\e 1\cdots\eps_{n+1}}
W_{\e 1}W_{\e 1\e 2}\cdots W_{\e 1\cdots\eps_{n+1}}\\
&=\frac1\ell\sum_{j=0}^{\ell-1}W_j\left[\ell^{-n}\sum_{\e 2\cdots\eps_{n+1}}W_{j\e 2}\cdots W_{j\cdots\eps_{n+1}}\right]
\end{aligned}
\end{equation}
and the sequence $(Y_n)$ is a solution in law of the equation
\begin{equation}\label{EQfonda}
Y_{n+1}=\frac1\ell\sum_{j=0}^{\ell-1}W_jY_n(j).
\end{equation}
where the $Y_n(0),\cdots Y_n(\ell-1)$ are independent copies of $Y_n$, and are independent to $W_0,\cdots ,W_{\ell-1}$.   

Taking the limit in the equality (\ref{EQiter}), the total mass $Y_\infty=m([0,1]$ is also a solution in law  of the equation
\begin{equation}\label{EQfonda2}
Y_\infty=\frac1\ell\sum_{j=0}^{\ell-1}W_jY_\infty(j)
\end{equation}
where $Y_\infty(0),\cdots Y_\infty(\ell-1)$ are independent copies of $Y_\infty$,  and are independent to $W_0,\cdots ,W_{\ell-1}$.

Equations (\ref{EQfonda}) and (\ref{EQfonda2}) are called the fundamental equations and will be very useful in the following.
\subsection{Non-degeneracy}
As proved in Theorem \ref{THEOexistence}, the sequence $Y_n=m_n([0,1])$ is a non-negative martingale and we only know in the general case that $E[Y_\infty]\le 1$. In particular, the situation where $E[Y_\infty]=0$ is possible and is called the degenerate case. The first natural problem related to the random measure $m$ is then to find conditions that ensure that $m$ is not almost surely equal to 0 (i.e. $E[Y_\infty]\not= 0$). An abstract answer is given by an equi-integrability property. We will see further a more concrete necessary and sufficient condition (Theorem \ref{THEOequi}) and more concrete sufficient conditions (Proposition \ref{PROPl2} and Theorem \ref{THEOmoments}).
\begin{proposition}\label{PROPequi}
 Let $m$ be a Mandelbrot cascade associated to a weight $W$. Denote as before $Y_n=m_n([0,1])$ and $Y_\infty=m([0,1])$. The following are equivalent
 \begin{enumerate}
 \item $E[Y_\infty]=1$
\item $E[Y_\infty]>0$\quad(i.e. $P[m([0,1])\not=0]>0$)
\item The martingale $(Y_n)$ is equi-integrable
\end{enumerate}
In that case, we say that the Mandelbrot cascade $m$ is non-degenerate.
\end{proposition}
\begin{proof}
Suppose that 2. is true. Considering $Z=\frac{Y_\infty}{E[Y_\infty]}$, it follows that the fundamental equation
 $$Z=\frac1\ell\sum_{j=0}^{\ell-1}W_jZ(j)$$
has a solution satisfying $E[Z]=1$.

Iterating the fundamental equation, we get 
$$Z=\frac1{\ell^n}\sum_{\e 1\cdots\e n\in\m n}W_{\e 1}\cdots W_{\e 1\cdots\e n}Z(\e 1\cdots\e n)$$
in which the $Z(\e 1\cdots\e n)$ are copies of $Z$ independent to the $W_\eps$. Let $\a n$ be again the $\sigma$-algebra generated by the $W_\eps$, $\eps\in\m 1\cup\cdots\cup\m n$. We get 
$$E[Z\vert \a n]=\frac1{\ell^n}\sum_{\e 1\cdots\e n\in\m n}W_{\e 1}\cdots W_{\e 1\cdots\e n}E[Z(\e 1\cdots\e n)] =Y_n$$
and the martingale $(Y_n)$ is equi-integrable.

The proof of $3\Rightarrow 1$ is elementary. Indeed, if $(Y_n)$ is equi-integrable, it converges to $Y_\infty$ in $L^1$. In particular, $\displaystyle E[Y_\infty]=\lim_{n\to \infty}E[Y_n]=1$. 
\end{proof}

\begin{remark}
In fact, the proof of Proposition \ref{PROPequi} says that the condition of non degeneracy of the cascade $m$ is equivalent to the existence of a non negative solution $Z$ satisfying $E[Z]=1$ for the fundamental equation 
\begin{eqnarray}\label{EQNfond}
Z=\frac1\ell\sum_{j=0}^{\ell-1}W_jZ(j).
\end{eqnarray}
\end{remark}

\begin{remark}
Equation (\ref{EQNfond}) may have non-integrable solutions. For example, if $\ell=2$ and $W=1$, equation (\ref{EQNfond}) becomes
 $$Z=\frac12(Z(1)+Z(2)).$$
If $Z(1)$ et $Z(2)$ are two independent Cauchy variables (with density  $\frac{dz}{\pi(1+z^2)}$), then $Z$ is also a Cauchy variable.
\end{remark}

 In the non-degenerate case, we only know that $P[m\not= 0]>0$ almost surely. A natural question is then to ask if $P[m\not= 0]=1$ almost surely. The answer to this question is easy.

\begin{proposition}
 Suppose that the Mandelbrot cascade $m$ is non-degenerate. Then,  
 $$P[m\not= 0]=1\quad\mbox{if and only if}\quad P[W=0]=0.$$
\end{proposition}
\begin{proof}
 Suppose that $(Y_n)$ is equi-integrable. Let us write again the fundamental equation

$$Y_\infty=\frac1\ell\sum_{j=0}^{\ell-1}W_jY_\infty(j).$$
Then 
\begin{eqnarray*}
P[Y_\infty=0]&=&P\left[W_0Y_\infty(0)=0\mbox{ and }\cdots\mbox{ and }W_{\ell-1}Y_\infty(\ell-1)=0\right]\\
&=&P[WY_\infty=0]^\ell\\
&=&\left(1-P[W\not=0\ \mbox{and}\ Y_\infty\not= 0]\right)^\ell
\end{eqnarray*}
If $r=P[W=0]$, it follows that $P[Y_\infty=0]$ is a fixed point of the function $$f(x)=(r+(1-r)x)^\ell.$$
We know that $P[Y_\infty=0]<1$. The second fixed point of the function $f$ is equal to 0 if and only if $r=0$. The conclusion follows.
\begin{center}
\includegraphics[scale=0.3]{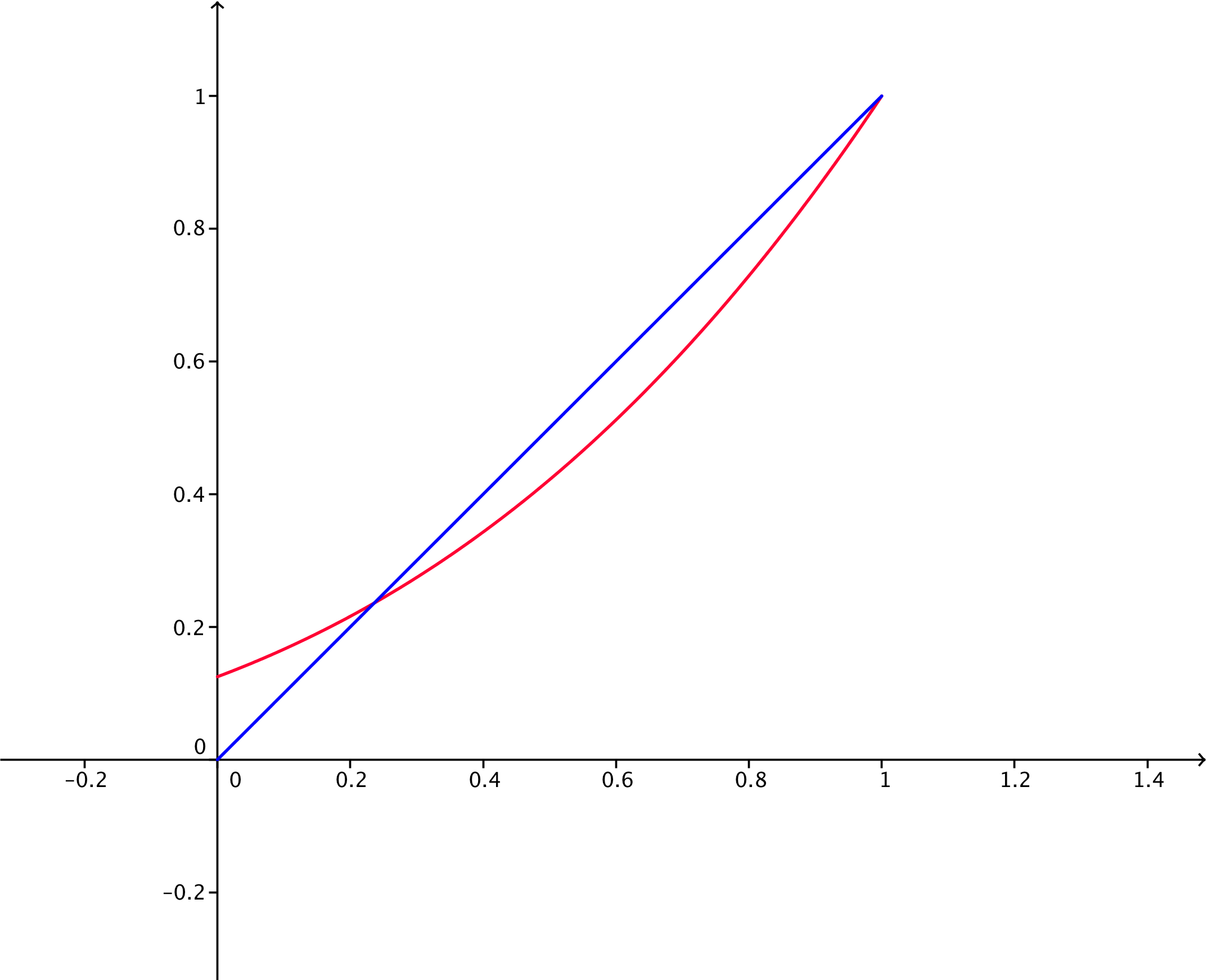}
\captionof{figure}{The graph of the function $f$}
\end{center}
\end{proof}

In the $L^2$ case it is easy to obtain a condition on the second order moment which gives non-degeneracy.

\begin{proposition}\label{PROPl2}
Suppose that $E[W^2]<+\infty$. The following are equivalent
\begin{enumerate}
\item $E[W^2]<\ell$
\item The sequence $(Y_n)$ is bounded in $L^2$
\item $0<E[Y_\infty^2]<+\infty$
\end{enumerate}
In particular, if 1. is true, the sequence $(Y_n)$ is equi-integrable and the cascade $m$ in non-degenerate.
\end{proposition}

\begin{proof}
 Let us write the fundamental equation
$$Y_{n+1}=\frac1\ell\sum_{j=0}^{\ell-1}W_jY_n(j).$$
We get
\begin{eqnarray*}
E[Y_{n+1}^2]&=&\frac1{\ell^2}\left(\sum_{j=0}^{\ell-1}E[(W_j^2Y_n(j))^2]+\sum_{i\not= j}E[W_iY_n(i)W_jY_n(j)]\right)\\
&=&\frac1\ell E[W^2]E[Y_n^2]+\frac1{\ell^2}\times\ell(\ell-1)\\
\end{eqnarray*}
It follows that the sequence $(E[Y_n^2])$ is bounded if and only if the common ratio $\frac1\ell E[W^2]$ is lower than 1. So 1. is equivalent to 2.

$2.\Rightarrow 3.$ Suppose that the sequence $(Y_n)$ is bounded in $L^2$. We know that the martingale $(Y_n)$ converges in $L^2$. In particular
$$E[Y_\infty^2]=\lim_{n\to +\infty}E[Y_n^2]<+\infty.$$
Moreover the sequence $(Y_n^2)$ is a submartingale and the sequence $(E[Y_n^2])$ is non-decreasing. It follows that $E[Y_\infty^2]>0$. which gives 3.

$3.\Rightarrow 1.$ Suppose that $0<E[Y_\infty^2]<+\infty$. According to Proposition \ref{PROPequi}, the martingale $(Y_n)$ is non-degenerate. In particular, $E[Y_\infty]=1$. The fundamental equation says that 
$$Y_\infty=\frac1\ell\sum_{j=0}^{\ell-1}W_jY_\infty(j).$$
It follows that 
\begin{eqnarray*}
E[Y_\infty^2]&=&\frac1{\ell^2}\left(\sum_{j=0}^{\ell-1}E[(W_j^2Y_\infty(j))^2]+\sum_{i\not= j}E[W_iY_\infty(i)W_jY_\infty(j)]\right)\\
&=&\frac1\ell E[W^2]E[Y_\infty^2]+\frac1{\ell^2}\times\ell(\ell-1)\\
\end{eqnarray*}
so that
$$\left(\ell-E[W^2]\right)E[Y_\infty^2]=\ell-1.$$
In particular, $E[W^2]<\ell$.
\end{proof}

A generalization of Proposition \ref{PROPl2} in the case where the weight $W$ admits an $L^q$ moment is possible. This is the object of Section \ref{SECmoments}. Nevertheless, we can also give a characterization on the non-degeneracy of the cascade $m$. It is given in terms of the $L\log L$ moment of the weight $W$. 
\begin{theorem}[Kahane, 1976, \cite{KP76}]\label{THEOequi}
 Let $m$ be a Mandelbrot cascade associated to a weight $W$. The following are equivalent
\begin{enumerate}
\item The cascade $m$ is non-degenerate
\item The martingale $(Y_n)$ is equi-integrable
\item $E[W\log W]<\log\ell$
\end{enumerate}
\end{theorem}

We begin with a geometric interpretation of the condition $E[W\log W]<\log\ell$. Let us introduce the structure function $\tau$, which is defined by 
\begin{equation}\label{EQtau}
\tau(q)=\log_\ell E\left[\sum_{j=0}^{\ell-1}\left[\frac1\ell W_j\right]^q\right]
=\log_\ell\left(E[W^q]\right)-(q-1).
\end{equation}
Such a formula makes sense when $0\le q\le 1$ (and perhaps for other values of $q$) and we always use the convention $0^q=0$. In particular, 
$$\tau(0)=1+\log_\ell(P[W\not= 0])$$
which will be seen as the almost sure Hausdorff dimension of the closed support of the measure $m$. 

The function $\tau$ is continuous and convex on $[0,1]$ and we will show that 
$$\tau'(1^-)=E[W\log_\ell W]-1\le +\infty.$$
It follows that Condition 3 in Theorem \ref{THEOequi} is equivalent to $\tau'(1^-)<0$.

Set $\phi(q)=E[W^q]$. In order to prove that $\tau'(1^-)=E[W\log_\ell W]-1$, we have to understand why we can write $\phi'(1^-)=E[W\log W]$, with a possible value equal to $+\infty$. Indeed, using the dominated convergence theorem, we have $\phi'(q)=E[W^q\log W]$ when $0\le q<1$. On one hand, the convexity of the function $\phi$ allows us to write $$\displaystyle\lim_{q\to1^-}\phi'(q)=\phi'(1^-)\le +\infty.$$
On the other hand, 
$$\phi'(q)=E[W^q\log W]=E[W^q\log W\,\un_{\{W<1\}}]+E[W^q\log W\,\un_{\{W\ge 1\}}].$$
The non-negative quantity $E[W^q\log W\,\un_{\{W\ge 1\}}]$ increases to $E[W\log W\,\un_{\{W\ge 1\}}]$ and by the dominated convergence theorem, the quantity $E[W^q\log W\,\un_{\{W<1\}}]$ goes to $E[W\log W\,\un_{\{W<1\}}]$. The formula $\phi'(1^-)=E[W\log W]$ follows.
\begin{proof}[Proof of Theorem \ref{THEOequi}]
According to Proposition \ref{PROPequi}, we just have to prove that Conditions 2 and 3 are equivalent.
\vskip 0.3cm
Step 1. {\it $\tau'(1^-)\le 0$ is a necessary condition.}\\

\noindent
Suppose that  the sequence $(Y_n)$ is equi-integrable. Then, the fundamental equation
$$Z=\frac1\ell\sum_{j=0}^{\ell-1}W_jZ(j)$$
has a non-negative solution with expectation equal to 1.
If $0<q\le 1$, the function  $x\mapsto x^q$ is subadditive (that is satisfies $(a+b)^q\le a^q+b^q$). We get 
$$E\left[\ell^qZ^q\right]\le\sum_{j=0}^{\ell-1}E[W_j^qZ(j)^q]=\ell E[W^q]E[Z^q].$$
Observe that $E[Z^q]>0$, so that 
$$\ell^q\le\ell E[W^q].$$
Finally, $\tau(q)\ge0$ if $q\le 1$ and $\tau'(1^-)\le 0$.
\vskip 0.3cm
Step 2. {\it More precisely, $\tau'(1^-)< 0$ is a necessary condition.}\\

\noindent
We have to improve the previous result. We need a lemma which gives a more precise estimate than the subadditivity of the function $x\mapsto x^q$.
\begin{lemma}\label{LEMsub}
If  $0<q<1$ and if  $0<y\le x$, then \quad$(x+y)^q\le x^q+qy^q$.
\end{lemma}
\begin{proof}Using homogeneity, we may assume that $y=1$ and $x\ge 1$. The inequality $(x+1)^q-x^q\le q$ is then an easy consequence of the mean value theorem. 
\end{proof}
We also need the following elementary lemma on random variables.
\begin{lemma}
 Let $X$ and $X'$ be two non-negative i.i.d. random variables such that $E[X]>0$. There exists $\delta>0$ such that for any $q\in[0,1]$,
$E[X^q\un_{X'\ge X}]\ge\delta E[X^q]$.
\end{lemma}
\begin{proof}
We claim that for any $q\in[0,1]$, $E[X^q\un_{X'\ge X}]>0$. Indeed, if $E[X^q\un_{X'\ge X}]=0$ for some $q$, then $X$ is almost surely equal to 0 on the set $\{X'\ge X\}$. By symmetry, $X'$ is almost surely equal to 0 on the set $\{X\ge X'\}$. Then $XX'=0$ almost surely, which is in contradiction with $E[XX']=E[X]E[X']>0$. 
Moreover, the functions $q\mapsto E[X^q\un_{X'\ge X}]$ and $q\mapsto E[X^q]$ are continuous on $[0,1]$ and the conclusion follows. 
\end{proof}

We can now prove that $\tau'(1^-)< 0$ is a necessary condition. Let 
$$A=\{W_1Z(1)\ge W_0Z(0)\}.$$
Using subadditivity of $x\mapsto x^q$ and Lemma \ref{LEMsub}, we have :
$$
\left\{^{\displaystyle (\ell Z)^q\le\sum_{j=0}^{\ell-1}(W_jZ(j))^q}
_{\displaystyle(\ell Z)^q\le q(W_0Z(0))^q+\sum_{j=1}^{\ell-1}(W_jZ(j))^q\qquad\mbox{on }A.}\right.
$$
Then,
\begin{eqnarray*}
E\left[(\ell Z)^q\right]&=&E\left[(\ell Z)^q\un_{A}\right]+E\left[(\ell Z)^q\un_{A^c}\right]\\
&\le& qE\left[(W_0Z(0))^q\un_{A}\right]+\sum_{j=1}^{\ell-1}E\left[(W_jZ(j))^q\un_A\right]+\sum_{j=0}^{\ell-1}E\left[(W_jZ(j))^q\un_{A^c}\right]\\
&=&(q-1)E\left[(W_0Z(0))^q\un_{A}\right]+\ell E[W^q]E[Z^q]\\
&\le&(q-1)\delta E[W^q]E[Z^q]+\ell E[W^q]E[Z^q].
\end{eqnarray*}
We get
$$\ell^{1-q}E[W^q]\ge\frac1{1+(q-1)\frac\delta\ell}$$
so that
$$\tau(q)\ge -\log_\ell\left(1+(q-1)\frac{\delta}\ell\right).$$
Finally, 
$$\tau'(1^-)\le-\frac{\delta}{\ell\log\ell}<0.$$
\vskip 0.3cm
Step 3. {\it $\tau'(1^-)< 0$ is a sufficient condition.}\\

\noindent
We suppose that $E[W\log W]<\log\ell$  (i.e. $\tau'(1^-)<0$) and, according to Proposition \ref{PROPequi}, we want to prove that $E[Y_\infty]>0$. Now, we need a precise lower bound of quantities such as $\left(\sum_{j=1}^\ell x_j\right)^q$. We will use the following lemma.
\begin{lemma}\label{LEMMAsur}
If $x_1\cdots x_\ell\ge 0$, and if $0<q\le1$, then
\begin{equation}\label{EQsur}
\left(\sum_{j=1}^\ell x_j\right)^q\ge\sum_{j=1}^\ell x_j^q-2(1-q)\sum_{i<j}(x_ix_j)^{q/2}.
\end{equation}
\end{lemma}
Suppose first that the lemma is true and let us write again the fundamental equation
$$\ell Y_n=\sum_{j=0}^{\ell-1}W_jY_{n-1}(j).$$
Lemma \ref{LEMMAsur} ensures that 
$$(\ell Y_n)^q\ge\sum_{j=0}^{\ell-1}(W_jY_{n-1}(j))^q-2(1-q)\sum_{i<j}(W_iY_{n-1}(i)W_jY_{n-1}(j))^{q/2}.$$
Taking the expectation and using that $Y_n^q$ is a supermartingale, we get
\begin{eqnarray*}
\ell^qE[Y_n^q]&\ge&\ell E[W^q]E\left[Y_{n-1}^q\right]-\ell(\ell-1)(1-q)E[W^{q/2}]^2\times E\left[Y_{n-1}^{q/2}\right]^2\\
&\ge&\ell E[W^q]E\left[Y_{n}^q\right]-\ell(\ell-1)(1-q)E[W^{q/2}]^2\times E\left[Y_{n-1}^{q/2}\right]^2
\end{eqnarray*}
Finally,
\begin{eqnarray*}
 E\left[Y_{n}^q\right](\ell^{\tau(q)}-1)&=&E\left[Y_{n}^q\right](\ell^{1-q}E[W^q]-1)\\
&\le&\ell^{1-q}(\ell-1)(1-q)E\left[Y_{n-1}^{q/2}\right]^2\times E\left[W^{q/2}\right]^2\\
&\le&\ell^{1-q}(\ell-1)(1-q)E\left[Y_{n-1}^{q/2}\right]^2\times E\left[W^q\right].
\end{eqnarray*}
Dividing by $1-q$ and taking the limit when $q$ goes to $1^-$, we get
$$1\times(-\tau'(1^-)\times\log\ell)\le(\ell-1)E\left[Y_{n-1}^{1/2}\right]^2\times 1$$
which gives that $E\left[Y_{n-1}^{1/2}\right]\ge C>0$. Observing that the supermartingale $\left(Y_n^{1/2}\right)$ converges almost surely to $Y_\infty^{1/2}$ and is bounded in $L^2$, we conclude that $\left(Y_n^{1/2}\right)$ is equi-integrable and converges in $L^1$. In particular,
$$E[Y_\infty^{1/2}]=\lim_{n\to+\infty} E[Y_n^{1/2}]\ge C.$$
So $E[Y_\infty]>0$ and the cascade $m$ is non-degenerate.
\end{proof}

Let us now finish this part with the proof of Lemma \ref{LEMMAsur}. Suppose first that $\ell=2$. By homogeneity the inequality is equivalent to $$\left(x+x^{-1}\right)^q\ge x^q+x^{-q}-2(1-q)$$
for any $x>0$. Let 
$$\varphi(x)=x^q+x^{-q}-\left(x+x^{-1}\right)^q.$$ 
If $0<x\le 1$, we have 
\begin{eqnarray*}
\varphi'(x)&=&qx^{-(q+1)}\left[ x^{2q}-1+\left(1-x^2\right)\left(1+x^2\right)^{q-1}\right]\\
&\ge&qx^{-(q+1)}\left[ x^{2q}-1+\left(1-x^2\right)\left(1+(q-1)x^2\right)\right]\\
&=&qx^{-(q+1)}\left[ x^{2q}+(q-2)x^2-(q-1)x^4\right].
\end{eqnarray*}
By studying the function $\psi(y)=y^q+(q-2)y-(q-1)y^2$, it is then easy to see that $\psi(y)\ge0$ for any $y\in[0,1]$.

Finally, for any $x>0$, 
$$\varphi(x)=\varphi(x^{-1})\le\varphi(1)=2-2^q\le2\ln2(1-q)\le2(1-q).$$
and the proof is done in the case $\ell=2$. 

The general case is easily obtained by induction on $\ell$, using once again that the function $x\mapsto x^{q/2}$ is subadditive if $0<q<1$.
\begin{remark}
In fact, the proof of Lemma \ref{LEMMAsur} says that  the constant $-2(1-q)$ in (\ref{EQsur}) can replaced by $-2\ln 2(1-q)$ which is the optimal one.
\end{remark}
\begin{example}[Birth and death processes] Suppose that $dP_W=(1-p)\delta_0+p\delta_{\frac1{p}}$. Then 
$$E\left[ W^q\right]=0P[W=0]+\left(\frac1{p}\right)^qP\left[W=\frac1{p}\right]=p^{1-q}$$
and
$$\tau(q)=\log_\ell(E[W^q])-(q-1)=(1-q)\times(1+\log_\ell p).$$
The cascade is non-degenerate if and only if $p>1/\ell$, that is if and only if $P[W=0]<1-\frac1\ell$. In that case, the box dimension of the closed support of the measure $m$ is almost surely $d=\tau(0)=1+\log_\ell p$ on the set $\{ m\not= 0\}$.
\end{example}
\begin{example}[Log-normal cascades]
Suppose that
$$W=e^{\sigma N-\sigma^2/2}$$
where $N$ follows a standard normal distribution.
\begin{eqnarray*}
E[W^q]&=&\int e^{q(\sigma x-\sigma^2/2)}e^{-x^2/2}\,\frac{dx}{\sqrt{2\pi}}\\
&=&\int e^{-(x-q\sigma)^2/2}e^{q^2\sigma^2/2}e^{-q\sigma^2/2}\,\frac{dx}{\sqrt{2\pi}}\\
&=&e^{q^2\sigma^2/2}e^{-q\sigma^2/2}
\end{eqnarray*}
and
$$\tau(q)=\log_\ell(E[W^q])-(q-1)=\frac{\sigma^2}{2\ln\ell}(q^2-q)-(q-1).$$
The cascade is non-degenerate if and only if  $\sigma^2<2\log\ell$
\begin{center}
\includegraphics[scale=0.4]{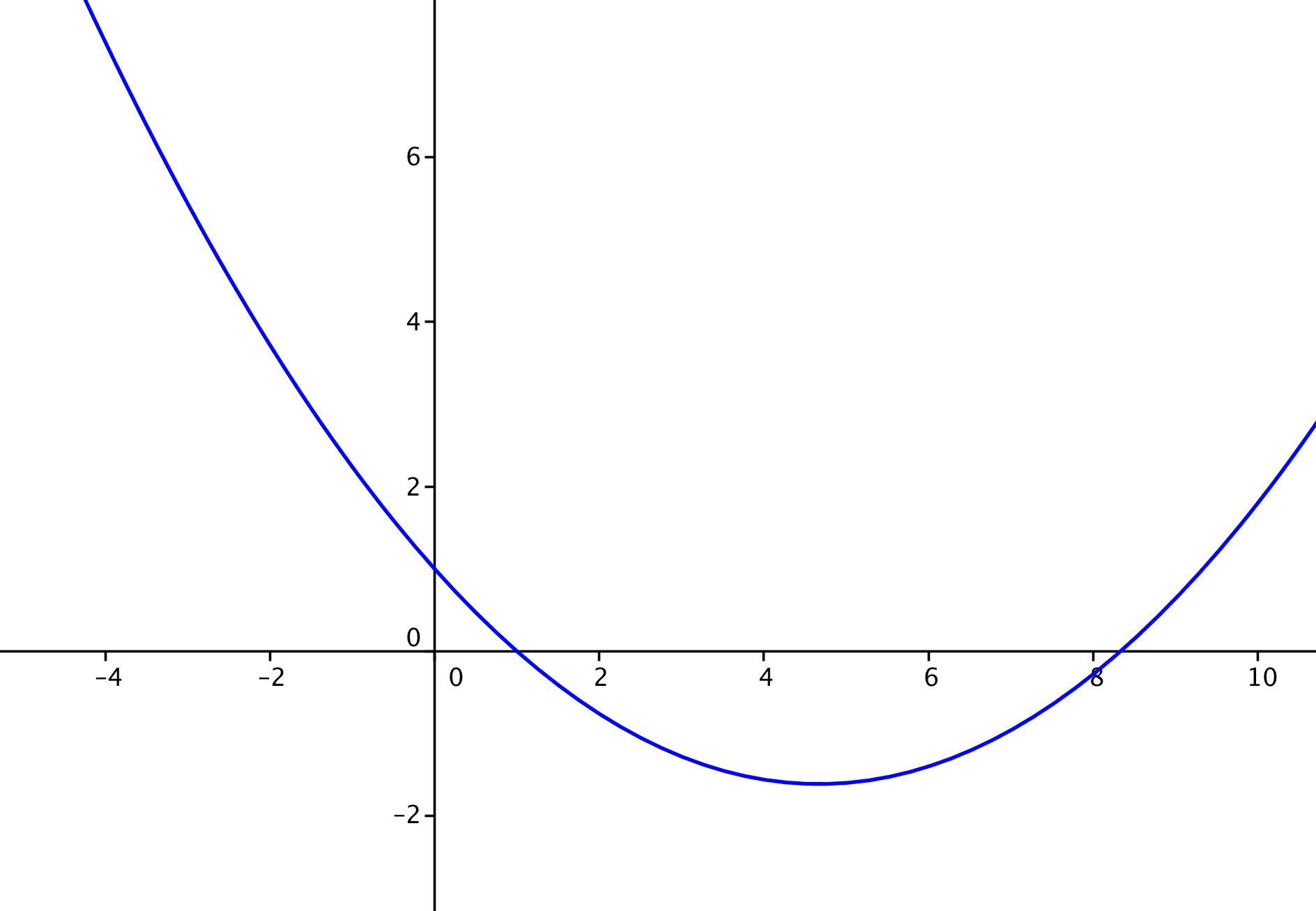}
\captionof{figure}{The structure function $\tau$ for a non-degenerate log-normal cascade}\label{FIGlognorm}
\end{center}
\end{example}
\section{The problem of moments}\label{SECmoments}
In Proposition \ref{PROPl2}, we obtained a necessary and sufficient condition for the martingale $(Y_n)$ to be bounded in $L^2$. This condition can be generalized in the following way.
\begin{theorem}[Kahane, 1976, \cite{KP76}]\label{THEOmoments}
Let $q>1$. Suppose that $E[W^q]<+\infty$. \\The following are equivalent
\begin{enumerate}
\item $E[W^q]<\ell^{q-1}$ (i.e. $\tau(q)<0$)
\item The sequence $(Y_n)$ is bounded in $L^q$
\item $0<E[Y_\infty^q]<+\infty$
\end{enumerate}
In particular, if 1. is true, the sequence $(Y_n)$ is equi-integrable and the cascade $m$ is non-degenerate.
\end{theorem}
\begin{remark}
 The condition $E[W^q]<\ell^{q-1}$ is equivalent to $\tau(q)<0$. The graph of the function $\tau$ allows us determine  the set of values of $q>1$ such that $(Y_n)$ is bounded in $L^q$ (see Figure \ref{FIGlognorm} for the case of log-normal cascades).
\end{remark}
\begin{proof}[Proof of Theorem \ref{THEOmoments}]

$2.\Rightarrow 3.$ If $(Y_n)$ is bounded in $L^q$, the martingale $(Y_n)$  converges in $L^q$. In particular, 
$$E[Y_\infty^q]=\lim_{n\to +\infty}E[Y_n^q]<+\infty.$$
Moreover the sequence $(Y_n^q)$ is a submartingale and the sequence $(E[Y_n^q])$ is non-decreasing. It follows that $E[Y_\infty^q]>0$. which gives 3.

$3.\Rightarrow 1.$ Suppose that $0<E[Y_\infty^q]<+\infty$ and write the fundamental equation
$$\ell Y_\infty=\sum_{j=0}^{\ell-1}W_jY_\infty(j).$$
Using the superaddititivity of the function $x\mapsto x^q$, we get
$$E[\ell^qY_\infty^q]\ge\sum_{j=0}^{\ell-1}E\left[(W_jY_\infty(j))^q\right]=\ell E[W^q]E[Y_\infty^q]$$
and the equality case is not possible. In particular, $E[W^q]<\ell^{q-1}$.

$1.\Rightarrow 2.$ This is the difficult part of the theorem. Let us begin with the easier case $1<q\le 2$. Recall once again the fundamental equation
$$\ell Y_{n+1}=\sum_{j=0}^{\ell -1}W_jY_n(j).$$
The function $x\mapsto x^{q/2}$ is sub additive so that
\begin{eqnarray*}
(\ell Y_{n+1})^q&\le& \left(\sum_{j=0}^{\ell -1}(W_jY_n(j))^{q/2}\right)^2\\
&=& \sum_{j=0}^{\ell -1}(W_jY_n(j))^q +\sum_{i\not=j}(W_iY_n(i))^{q/2}(W_jY_n(j))^{q/2}
\end{eqnarray*}
Taking the expectation, and using that $(Y_n^q)$ is a submartingale, we get
\begin{eqnarray*}
\ell^qE\left[ Y_{n+1}^q\right]&\le&\ell E\left[ Y_{n}^q\right]E\left[W^q\right]+\ell(\ell-1)E[W^{q/2}]^2E\left[ Y_n^{q/2}\right]^2\\
&\le&\ell E\left[ Y_{n}^q\right]E\left[W^q\right]+\ell(\ell-1)E[W]^qE[Y_n]^q\\
&=&\ell E\left[ Y_{n+1}^q\right]E\left[W^q\right]+\ell(\ell-1)
\end{eqnarray*}
Finally,
\begin{eqnarray}\label{EQsimple}
E[Y_{n+1}^q]\le \frac{\ell-1}{\ell^{q-1}-E[W^q]}.
\end{eqnarray}
Suppose now that $k<q\le k+1$ where $k\ge 2$ is an integer and write
$$
(\ell Y_{n+1})^q\le \left(\sum_{j=0}^{\ell -1}(W_jY_n(j))^{q/(k+1)}\right)^{k+1}
= \sum_{j=0}^{\ell -1}(W_jY_n(j))^q +T
$$
where the quantity $T$ is a sum of $\ell^{k+1}-\ell$ terms of the form 
$$(W_{j_1}Y_n(j_1))^{\alpha_1q/(k+1)}\times\cdots\times(W_{j_p}Y_n(j_p))^{\alpha_pq/(k+1)}$$
with $p\ge 2$ and $\alpha_1+\cdots+\alpha_p=k+1$. The expectation of such a term satisfies
\begin{eqnarray*}
&&E\left[(W_{j_1}Y_n(j_1))^{\alpha_1q/(k+1)}\times\cdots\times(W_{j_p}Y_n(j_p))^{\alpha_pq/(k+1)}\right]\\
&\le&
E\left[(W_{j_1}Y_n(j_1))^k\right]^{\alpha_1q/k(k+1)}\times\cdots\times E\left[(W_{j_p}Y_n(j_p))^k\right]^{\alpha_pq/k(k+1)}\\
&=&\left(E\left[W^k\right]E\left[Y_n^k\right]\right)^{q/k}
\end{eqnarray*}
so that
$$\ell^qE\left[Y_{n+1}^q\right]\le \ell E\left[Y_n^q\right]\,E\left[W^q\right] +(\ell^{k+1}-\ell)\left(E\left[W^k\right]E\left[Y_n^k\right]\right)^{q/k}.$$
Using that $(Y_n^q)$ is a submartingale, we get
\begin{eqnarray}\label{EQprocheproche}
E\left[ Y_{n+1}^q\right](\ell^{q-1}-E[W^q])\le( \ell^k-1)\left(E[W^k]E\left[Y_n^k\right]\right)^{q/k}
\end{eqnarray}
which is the generalization of (\ref{EQsimple}). It follows that $(Y_n)$ is bounded in $L^q$ as soon as $(Y_n)$ is bounded in $L^k$.

Let us finally observe that the hypothesis $E[W^q]<\ell^{q-1}$ (i.e.$\tau(q)<0$) implies that $E[W^t]<\ell^{t-1}$ (i.e.$\tau(t)<0$) for any $t$ such that $1<t<q$. Replacing $q$ by $j+1$ in (\ref{EQprocheproche}), we also have
\begin{eqnarray*}
E\left[ Y_{n+1}^{j+1}\right](\ell^{j}-E[W^{j+1}])\le( \ell^j-1)\left(E[W^j]E\left[Y_n^j\right]\right)^{q/j}
\end{eqnarray*}
for any integer $j$ such that $2\le j<k$.  
 Step by step we get that $(Y_n)$ is bounded in $L^2$, $L^3$,..., $L^k$, $L^q$.
\end{proof}
\section{On the dimension of non-degenerate cascades}\label{SECdim}
The Mandelbrot cascade is almost-surely a unidimensional measure as was proved by Peyri\`ere in \cite{KP76}.
\begin{theorem}[Peyri\`ere, 1976, \cite{KP76}]\label{THEOdim}
 Suppose that $0<E[Y_\infty\log Y_\infty]<+\infty$. Then, almost surely,
$$\lim_{n\to +\infty}\frac{\log m(I_n(x))}{\log \vert I_n(x)\vert}=1-E[W\log_\ell W]\quad dm-\mbox{almost every where}$$
\end{theorem}
Let us recall that it is possible that $m=0$ with positive probability. So, the good way to rewrite Theorem \ref{THEOdim} is

\begin{corollary}\label{COROdim}
Suppose that $0<E[Y_\infty\log Y_\infty]<+\infty$. Almost surely on $\{ m\not=0\}$ we have :
\begin{enumerate}
 \item There exists a Borel set $E$ such that 
$$\dim(E)=1-E[W\log_\ell W]\quad\mbox{ and }\quad m([0,1]\setminus E)=0$$
\item If $\dim(F)<1-E[W\log_\ell W]$, then $m(F)=0$.
\end{enumerate}
It follows that $$\dim_*(m)=\dim^*(m)=1-E[W\log_\ell W]$$
where $\dim_*(m)$ and $\dim^*(m)$ are respectively the lower and the upper dimension of the measure $m$ as defined on \em{(\ref{EQdim})}.
\end{corollary}

\begin{remark}
 The condition $0<E[Y_\infty\log Y_\infty]<+\infty$ is stronger than $E[W\log W]<\log\ell$ which ensures the non-degeneracy of the cascade $m$. Indeed, suppose that $0<E[Y_\infty\log Y_\infty]<+\infty$. The superadditivity of the function $t\mapsto t\log t$ and the fundamental equation imply that
$$
\ell Y_\infty\log (\ell Y_\infty)\ge \sum_{j=0}^{\ell-1}(W_jY_\infty(j))\log(W_jY_\infty(j)).
$$
Taking the expectation,
$$E[\ell Y_\infty\log(\ell Y_\infty)]\ge\ell E[(WY_\infty)\log(WY_\infty)]$$
and the equality case is not possible. We get
$$E[Y_\infty\log(\ell Y_\infty)]>E[W\log W]E[Y_\infty]+E[Y_\infty\log Y_\infty]E[W]$$
so that
$$E[Y_\infty]\log\ell>E[W\log W]E[Y_\infty].$$
It follows that $E[W\log W]<\log \ell$ and the cascade is non-degenerate.
\end{remark}
\begin{remark}
 Under the hypothesis of Theorem \ref{THEOdim} and Corollary \ref{COROdim}, we have the following relation : 
$$\dim (m)=1-E[W\log_\ell W]=-\tau'(1^-).$$
\begin{center}
\includegraphics{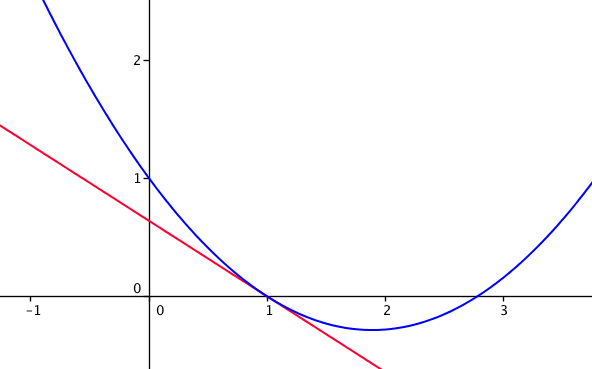}
\captionof{figure}{$\dim (m)=-\tau'(1)$}\label{FIGdim}
\end{center}
In particular, $0<\dim (m)\le 1$ almost surely on the event $\{ m\not= 0\}$.
\end{remark}
\begin{example}[Birth and death processes] Suppose that $dP_W=(1-p)\delta_0+p\delta_{\frac1{p}}$ with $p>1/\ell$. Then, 
$$\tau(q)=(1-q)\times(1+\log_\ell p)\qquad\mbox{and}\qquad\dim m=1+\log_\ell p$$
almost surely on the event $\{m\not= 0\}$.

Let $\mathcal{M}^*_n$ be the set of words $\eps\in\M_n$ such that $m(I_\eps)>0$. The closed support of $m$ is nothing else but the Cantor set
$$\mbox{supp}\,(m)=K=\bigcap_{n\ge 1}\bigcup_{\eps\in\mathcal{M}^*_n}\overline{I_\eps}.$$
Theorem \ref{THEOdim} and Corollary \ref{COROdim} ensure that almost surely on the event $\{m\not= 0\}$, the Hausdorff dimension of $K$ satisfies $\dim (K)\ge 1+\log_\ell p$ which is also known as the box dimension of $K$. finally, 
$$\dim (K)=1+\log_\ell p$$
almost surely on the event $\{m\not= 0\}=\{ K\not=\emptyset\}$.
\end{example} 
\begin{example}[Log-normal cascades] Suppose that $N$ follows a standard normal distribution and 
$W=e^{\sigma N-\sigma^2/2}$ with $\sigma^2<2\log\ell$. We know that
$$\tau(q)=\frac{\sigma^2}{2\log\ell}(q^2-q)+1-q$$ and we find
$$\dim m=1-\frac{\sigma^2}{2\log\ell}\quad \mbox{almost surely}.$$
\end{example} 
\begin{proof}[Proof of Theorem \ref{THEOdim}]
As observed before, under the hypothesis $$0<E[Y_\infty\log Y_\infty]<+\infty,$$ 
the cascade $m$ is non-degenerate. In particular, $E[Y_\infty]=1$. 

We first need to precisely define the sentence "almost surely $dm$-almost every where". 
Let $\tilde\Omega=\Omega\times [0,1]$ endowed with the product $\sigma$-algebra. Define the measure $Q$ by 
$$Q[A]=E\left[\int\un_A\,dm\right].$$
Observe that the measure $m$ depends on $\omega\in \Omega$ so that $Q$ is not a product measure. Nevertheless,
$$Q[\tilde\Omega]=E\left[\int dm\right]=E[Y_\infty]=1$$
so that $Q$ is a probability measure. If a property is true on a set $A\subset\tilde\Omega $ satisfying $Q[A]=1$, then, almost surely, the properly is true $dm$-almost every where. 

Recall that the measure $m$ is constructed as the weak limit of the sequence $m_n=f_n\lambda$ where 
$$f_n=\sum_{\e 1\cdots\e n\in\m n}W_{\e 1}W_{\e 1\e 2}\cdots W_{\e 1\cdots\e n}\un_{I_{\e 1\cdots\e n}}.$$
The proof of Theorem \ref{THEOdim} is an easy consequence of the two following lemmas.
\begin{lemma}\label{LEMMA4}
 Suppose that $E[W\log W]<\log\ell$. Then, almost surely, 
$$\lim_{n\to +\infty}\frac{\log f_n(x)}{n}= E[W\log W]\qquad \mbox{for }dm-\mbox{almost every }x.$$
\end{lemma}
\begin{lemma}\label{LEMMA5}
 Suppose that $E[Y_\infty\log Y_\infty]<+\infty$. Let $\mu_n=\frac1{f_n}m$. Then, almost surely, 
$$\lim_{n\to +\infty}\frac{\log \mu_n(I_n(x))}{n}= -\log\ell\qquad \mbox{for }dm-\mbox{almost every }x.$$
\end{lemma}
Suppose first that Lemma \ref{LEMMA4} and Lemma \ref{LEMMA5} are true and recall that $dm=f_nd\mu_n$. The density $f_n$ is constant on any interval of the $n^{th}$ generation, so that
$$m(I_n(x))=\displaystyle \int_{I_n(x)}f_n(y)d\mu_n(y)=f_n(x)\mu_n(I_n(x)).$$
It follows that 
\begin{eqnarray*}
\frac{\log(m(I_n(x))}{\log\vert I_n(x)\vert}&=&\frac{\log f_n(x)+\log\mu_n(I_n(x))}{-n\log\ell}\\
&\to&-E[W\log_\ell W]+1
\end{eqnarray*}
almost surely $dm$-almost every where.
\begin{proof}[Proof of Lemma \ref{LEMMA4}] Let us write 
$$f_n=g_1\times\cdots\times g_n\qquad\mbox{where}\qquad g_n=\sum_{\e 1\cdots\e n\in\m n}W_{\e 1\cdots\e n}\un_{I_{\e 1\cdots\e n}}.$$
We get
$$\frac{\log f_n}n=\frac1n\sum_{k=1}^n\log g_k$$
and Lemma \ref{LEMMA4} will be a consequence of the strong law of large numbers in the space $\tilde\Omega$ associated to the probability $Q$. 

Let us calculate the law of the random variable $g_n=\sum_{\eps\in\M_n}W_\eps1\!\!\!1_{I_\eps}$. If $\E$ is the expectation related to the probability $Q$ and if $\phi$ is bounded and measurable,

\begin{eqnarray*}
\E[\phi(g_n)]&=&E\left[\int\sum_{\eps\in \M_n}\phi(W_{\eps})\un_{I_{\eps}}\,dm\right]\\
&=&\sum_{\eps\in\M_n}E\left[\phi(W_{\eps})m(I_{\eps})\right]
\end{eqnarray*}
Moreover, if $k\ge 0$, using the independence properties, 
\begin{eqnarray*}
E\left[\phi(W_{\eps})m_{n+k}(I_{\eps})\right]&=&\sum_{\alpha_1\cdots\alpha_k\in \m k}E\left[\phi(W_{\eps})\ell^{-(n+k)}W_{\e 1}\cdots W_{\eps}W_{\eps\alpha_1}\cdots W_{\eps\alpha_1\cdots\alpha_k}\right]\\
&=&\ell^{-n}E\left[\phi(W)W\right].
\end{eqnarray*}
Taking the limit, we get 
\begin{equation}\label{EQNlaw}
\E[\phi(g_n)]=E[\phi(W)W].
\end{equation}
Equation (\ref{EQNlaw}) remains true if $\phi$ is such that $E\left[\vert\phi(W)W\vert\right]<+\infty$. In particular, the random variables $g_n$ have the same law and $\log(g_n)$ are integrable with respect to $Q$.

The independence of the sequence $(g_n)$ is obtained in a similar way. If $\phi_1,\cdots,\phi_n$ are bounded and measurable, we can also write
\begin{eqnarray*}
\E\left[\phi_1(g_1)\cdots\phi_n(g_n)\right]&=&\sum_{\eps\in\M_n}E\left[\phi_1(W_{\e 1})\cdots\phi_n(W_{\e 1\cdots\e n})m(I_\eps)\right]\\
&=&\cdots\\
&=&E[\phi_1(W)W]\times\cdots\times E[\phi_n(W)W]\\
&=&\E\left[\phi_1(g_1)\right]\times\cdots\times\E\left[\phi_n(g_n)\right]
\end{eqnarray*}
and the independence follows. Finally, the strong law of large numbers gives
$$\lim_{n\to +\infty}\frac1n\sum_{k=1}^n\log g_k=\E[\log g_1]=E[W\log W]\quad dQ-\mbox{almost surely.}$$
\begin{remark}[On the importance of the order of the quantifiers]
Let $x\in [0,1]$ and $\e 1\cdots\e n\cdots$ such that $x\in I_{\e 1\cdots\e n}$ for any $n$. We have 
$$\frac{\log(f_n(x))}n=\frac1n(\log W_{\e 1}+\cdots+\log W_{\e 1\cdots\e n}).$$
Using the strong law of large numbers, we get:
$$\mbox{For any }x\in[0,1],\mbox{ almost surely,}\quad\lim_{n\to+\infty}\frac{\log(f_n(x))}{n}=E[\log W]$$
which is different of the conclusion of Lemma \ref{LEMMA4} !
\end{remark}
\end{proof}
\begin{proof}[Proof of Lemma \ref{LEMMA5}] Let us begin with a comment on the definition of the measure $\mu_n$. The function $f_n$ is constant on any interval of the $n^{th}$ generation. Moreover, if $f_n$ is equal to zero on some interval $I$ of the $n^{th}$ generation, then $m(I)=0$. Finally, $f_n\not= 0$ $dm$-almost surely and $\mu_n=\frac1{f_n}m$ is well defined. We can also write $m=f_n\mu_n$ and if $\eps=\e 1\cdots\e n\in\M_n$,
$$m(I_\eps)=W_{\e 1}\cdots W_{\e 1\cdots\e n}\,\mu_n(I_\eps).$$
We claim that $\mu_n(I_\eps)$ is independent to $W_{\eps_1},\cdots,W_{\e 1\cdots\e n}$ and has the same distribution as $\ell^{-n}Y_\infty$. Indeed, 
$$m_{n+k}=f_n[(g_{n+1}\cdots g_{n+k})\,d\lambda]$$
and
$$\mu_n=\lim_{k\to +\infty}(g_{n+1}\cdots g_{n+k})\,d\lambda .$$
In particular,
$$\mu_n(I_\eps)=\lim_{k\to +\infty}\int_{I_\eps}g_{n+1}(x)\cdots g_{n+k}(x)\,d\lambda(x)$$
is clearly independent to $W_{\eps_1},\cdots,W_{\e 1\cdots\e n}$ and an easy calculation gives that it has the same distribution as $\ell^{-n}Y_\infty$. 

Using the previous remark, we get
\begin{eqnarray*}
 \E\left[\left(\ell^n\mu_n(I_n(x))\right)^{-1/2}\right]&=&E\left[\ell^{-n/2}\int\mu_n(I_n(x))^{-1/2}dm(x)\right]\\
&=&\ell^{-n/2}\sum_{\eps\in\M_n} E\left[\mu_n(I_\eps)^{-1/2}m(I_\eps)\right]\\
&=&\ell^{-n/2}\sum_{\eps\in\M_n} E\left[W_{\e 1}\cdots W_{\e 1\cdots\e n}\right]E\left[\mu_n(I_\eps)^{1/2}\right]\\
&=&E\left[Y_\infty^{1/2}\right].
\end{eqnarray*}
It follows that
 $$\E\left[\sum_{n\ge 1}\frac1{n^2}\left(\ell^n\mu_n(I_n(x))\right)^{-1/2}\right]<+\infty.$$
 In particular,  $dQ$-almost surely, $\ell^n\mu_n(I_n(x))\ge 1/n^4$ if $n$ is large enough and we can conclude that almost surely,
 $$\liminf_{n\to +\infty}\frac{\log\left(\ell^n\mu_n(I_n(x))\right)}n\ge 0\quad dm-\mbox{almost every where.}$$
 In other words, almost surely,
 $$\liminf_{n\to +\infty}\left(\frac{ \log\left(\mu_n(I_n(x))\right)}n\right)\ge -\log\ell \quad dm-\mbox{almost every where.}$$
 
 We have now to prove that almost surely, 
 $$\limsup_{n\to +\infty}\left(\frac{ \log\left(\mu_n(I_n(x))\right)}n\right)\le -\log\ell \quad dm-\mbox{almost every where.}$$
Recall that $m(I_\eps)=W_{\e 1}\cdots W_{\e 1\cdots\e n}\,\mu_n(I_\eps)$ with independence properties. If $\alpha >0$,
\begin{eqnarray*}
Q[\ell^n \mu_n(I_n(x))>\alpha^n]&=&E\left[\int \un_{\left\{\ell^n\mu_n(I_n(x))>\alpha^n\right\}}(x)\,dm(x)\right]\\
&=&\sum_{\eps\in\M_n}E\left[\int_{I_\eps}\un_{\left\{\ell^n\mu_n(I_\eps)>\alpha^n\right\}}(x)\,dm(x)\right]\\
&=&\sum_{\eps\in\M_n}E\left[m(I_\eps)\un_{\left\{\ell^n\mu_n(I_\eps)>\alpha^n\right\}}\right]\\
&=&\sum_{\eps\in\M_n}E[W_{\e 1}\cdots W_{\e 1\cdots\e n}]E\left[\mu_n(I_\eps)\un_{\left\{\ell^n\mu_n(I_\eps)>\alpha^n\right\}}\right]\\
&=&E\left[Y_\infty\un_{\left\{ Y_\infty>\alpha^n\right\}}\right]
\end{eqnarray*}
In particular,
\begin{eqnarray*}
\sum_{n\ge 1}Q[\ell^n \mu_n(I_n(x))>\alpha^n]&=&E\left[\sum_{n\ge 1}Y_\infty\un_{\left\{ Y_\infty>\alpha^n\right\}}\right]\\
&\le&E\left[ Y_\infty\log_\alpha^+(Y_\infty)\right]\\
&<&+\infty
\end{eqnarray*}
Using Borel Cantelli's lemma, we get
$$dQ-\mbox{almost surely},\quad \ell^n \mu_n(I_n(x))\le\alpha^n\quad\mbox{if }n\mbox{ is large enough.}$$
In particular, almost surely,
$$\limsup_{n\to +\infty}\frac{\log\left(\ell^n\mu_n(I_n(x))\right)}n\le \alpha\quad dm-\mbox{almost every where}$$
and the conclusion is a consequence of the arbitrary value of $\alpha$.
\end{proof}
\end{proof}
\begin{remark}
 In the eighties, Kahane proved that the condition $0<E[Y_\infty\log Y_\infty]<+\infty$ is not necessary.
\end{remark}
\section{A digression on multifractal analysis of measures}\label{SECdigression}
In order to understand the approach developed in Section \ref{SECmulti}, let us recall some basic facts on multifractal analysis of measures. In this part, $m$ is a deterministic measure on $[0,1]$ with finite total mass. As usual, we define the structure function as 
$$\tau(q)=\limsup_{n\to+\infty}\frac1n\log_\ell\left(\sum_{I\in\f n}m(I)^q\right)$$
and we want to briefly recall the way to improve the formula
$$\dim(E_\beta)=\tau^*(\beta)$$
where  $$E_\beta=\left\{x\ ;\ \lim_{n\to \infty}\frac{\log m(I_n(x))}{\log\vert I_n(x)\vert}=\beta\right\}$$ 
and 
$$\tau^*(\beta)=\inf_{q\in\R}(q\beta+\tau(q))$$
is the Legendre transform of $\tau$.

The function $\tau$ is known to be a non-increasing convex function on $\R$ such that $\tau(1)=0$. Moreover, the right and the left derivative $-\tau'(1^+)$ and $-\tau'(1^-)$ are related to the dimensions of the measure $m$ which are defined in formula (\ref{EQdim}) and (\ref{EQDim}). In the general case, as we can see for example in \cite{H07}, we have 
 \begin{theorem}
$$-\tau'(1^+)\le\diminf(m)\le\Dimsup(m)\le-\tau'(1^-).$$
\end{theorem}
\vskip 1cm
\begin{center}
\includegraphics[scale=0.15]{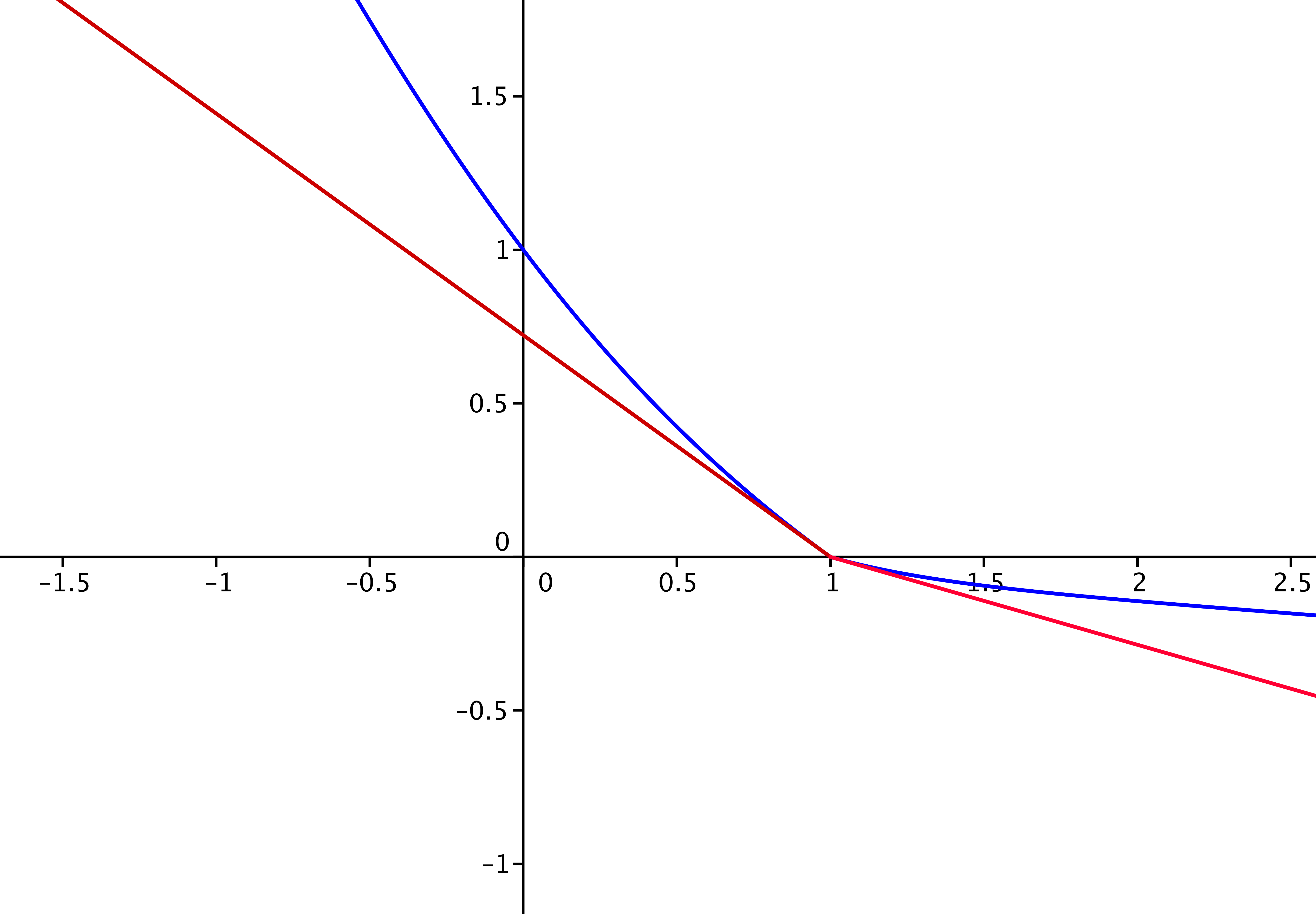}
\captionof{figure}{A case where $\tau'(1)$ does not exist}\label{FIGtau}
\end{center}
We can't ensure in general that $\tau'(1^+)=\tau'(1^-)$. Nevertheless, if $\tau'(1)$ exists, the measure $m$ is uni-dimensional and the following are true. 
\begin{corollary}
Suppose that $\tau'(1)$ exists. Then
\begin{enumerate}
\item $dm$-almost-surely, \quad$\displaystyle\lim_{n\to +\infty}\frac{\log(m(I_n(x)))}{\log\vert I_n(x)\vert}=-\tau'(1)$
\item $\dim\left(E_{-\tau'(1)}\right)=-\tau'(1)$
\item $\diminf(m)=\dimsup(m)=\Diminf(m)=\Dimsup(m)=-\tau'(1).$
\end{enumerate}
\end{corollary} 
The equality $\dim\left(E_{-\tau'(1)}\right)=-\tau'(1)$ can be rewritten in terms of the Legendre transform of the function $\tau$. More precisely, if $\beta=-\tau'(1)$, then $\tau^*(\beta)=\beta$ and $\dim(E_\beta)=\tau^*(\beta)$. This is the first step in mutlifractal formalism.

\begin{center}
\includegraphics[scale=0.15]{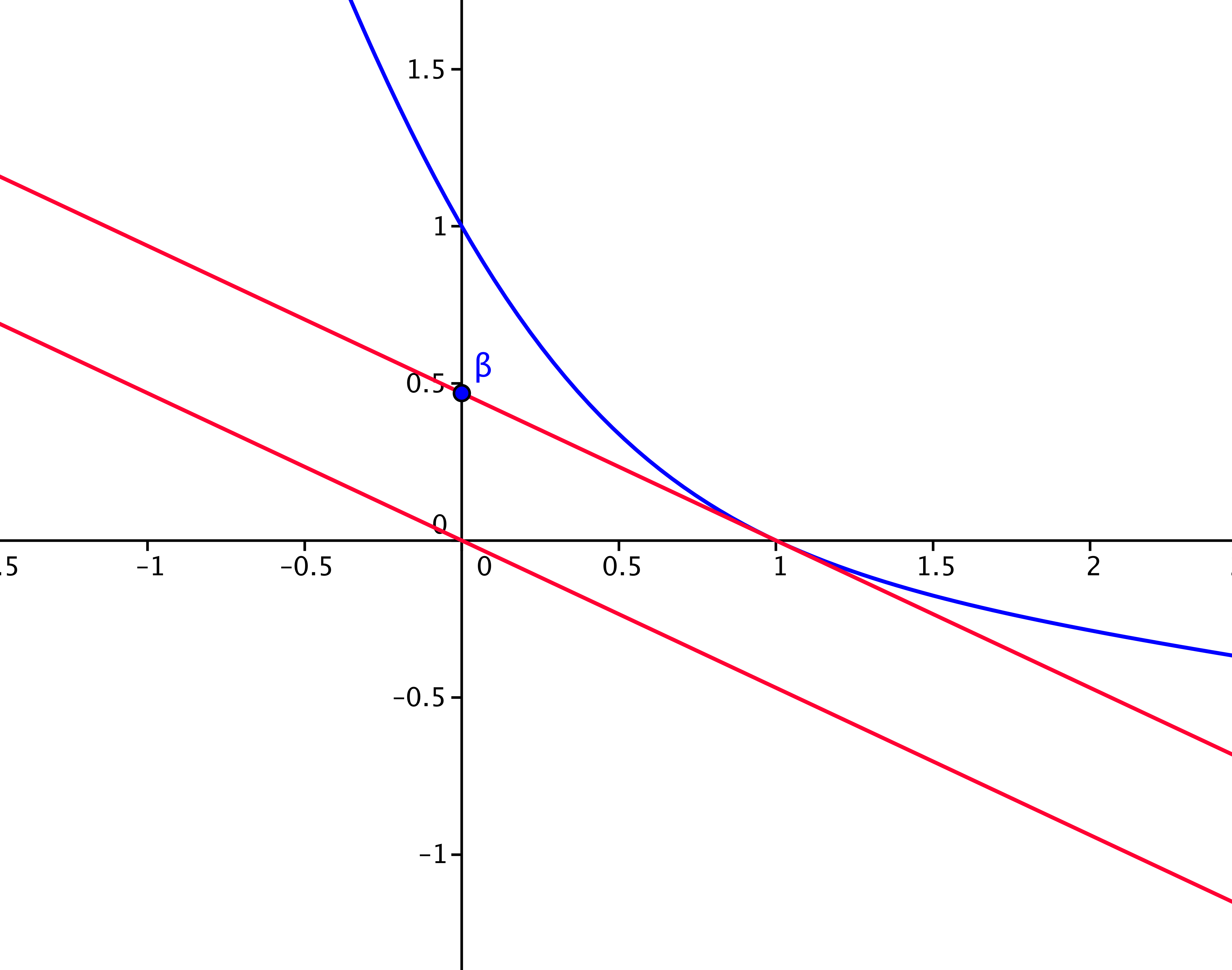}
\captionof{figure}{If $\beta=-\tau'(1)$, then $\tau^*(\beta)=\beta$}\label{FIGlegendre}
\end{center}
In order to obtain the formula $\dim(E_\beta)=\tau^*(\beta)$ for another value of $\beta$, the usual way is to write $\beta=-\tau'(q)$ and to construct an auxiliary measure $m_q$ (known as Gibbs measure) satisfying for any $\ell$-adic interval
$$\frac1C\,m(I)^q\vert I\vert^{\tau(q)}\le m_q(I)\le C\, m(I)^q\vert I\vert^{\tau(q)}.$$
This is the way used in \cite{BMP}. If such a measure $m_q$ exists, its structure function $\tau_q$ is such that 

$$\tau_q(t)=\tau(qt)-t\tau(q).$$
In particular, 
$$-\tau'_q(1)=-q\tau'(q)+\tau(q)=\tau^*(\beta).$$
If we observe that 
$$\frac{\log(m_q(I_n(x)))}{\log\vert I_n(x)\vert}=q\,\frac{\log(m(I_n(x)))}{\log\vert I_n(x)\vert}+\tau(q)+o(1),$$
we can conclude that
$$\dim(E_\beta)=\dim(m_q)=-\tau'_q(1)=\tau^*(\beta).$$
\section{Multifractal analysis of Mandelbrot cascades: an outline}\label{SECmulti}
In this section, we make the following additional assumptions:

\begin{equation}\label{EQsimplification}
\left\{
\begin{aligned}
& P[W=0]=0\\
\null\\
&\mbox{For any real } q,\quad E\left[W^q\right]<+\infty.
\end{aligned}
\right.
\end{equation}
In particular, assumption (\ref{EQsimplification}) is satisfied when $m$ is a log-normal cascade or when $\frac1C\le W\le C$  almost surely.

We can then list some  easy consequences.
\begin{itemize}
\item The function $\tau(q)=\log_\ell(E[W^q])-(q-1)$ is defined on  $\R$, convex and of class $C^\infty$
\item There exists $r>1$ such that $\tau(r)<0$\quad (and so $E[Y_\infty^r]<+\infty$)
\item The cascade is non-degenerate
\item $P[m= 0]=P[Y_\infty=0]=0$.
\end{itemize}

In order to perform the multifractal analysis of the Mandelbrot cascades, we want to mimic the way used for the binomial cascades. It is then natural to introduce the auxiliary cascade $m'$ associated to the weight
$W'=\frac{W^q}{E[W^q]}$ (the renormalization ensures that $E[W']=1$). The structure function of the cascade $m'$ is 
\begin{eqnarray*}
\tau_q(t)&=&\log_\ell( E[W'^t])-(t-1)\\
&=&\log_\ell\left(E\left[\frac{W^{qt}}{E[W^{q}]^t}\right]\right)-(t-1)\\
&=&\log_\ell\left(E\left[W^{qt}\right]\right)-t\log_\ell(E[W^q])-(t-1)\\
&=&\tau(tq)-t\tau(q)
\end{eqnarray*}
In particular, 
$$-\tau'_q(1^-)=-q\tau'(q)+\tau(q)=\tau^*(-\tau'(q))$$
and the cascade $m'$ is non-degenerate if and only if $\tau^*(-\tau'(q))>0$. This suggests to consider the interval
$$(q_{min},q_{max})=\{q\in\R\ ;\ \tau^*(-\tau'(q))>0\}.$$
\begin{example}[The interval $(q_{min},q_{max})$ in the case of log-normal cascades] 
If $m$ is a log-normal cascade, the function $\tau$ is given by
$$\tau(q)=\log_\ell(E[W^q])-(q-1)=\frac{\sigma^2}{2\ln\ell}(q^2-q)-(q-1)$$
and the numbers $q_{min}$ and $q_{max}$ are the solutions of the equation
$$\tau(q)=q\tau'(q).$$
We find
$$q_{min}=-\frac{\sqrt{2\ln \ell}}\sigma\qquad\mbox{and}\qquad q_{max}=\frac{\sqrt{2\ln \ell}}\sigma.$$
\begin{center}
\includegraphics[scale=0.3]{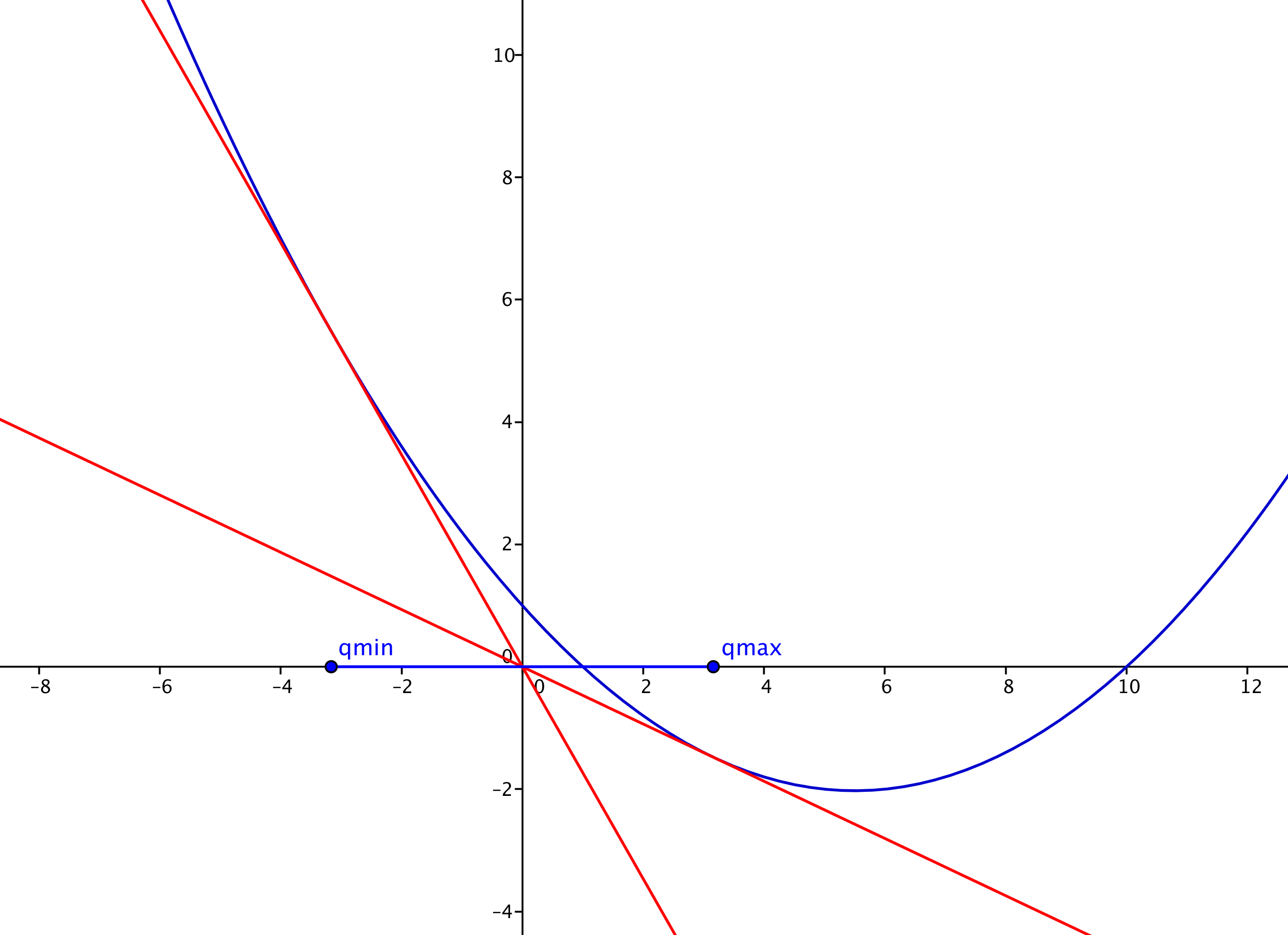}
\captionof{figure}{The interval $(q_{min},q_{max})$}\label{FIGqmin}
\end{center}
\end{example}

When $q\in(q_{min},q_{max})$, we would like to compare the behavior of the cascades $m$ and $m'$. In the following subsection, we give a general result which can be applied to the present situation.
\subsection{Simultaneous behavior of two Mandelbrot cascades}\label{SECsimultaneous}
\begin{theorem}\label{THEOsimultaneous}
Let $(W,W')$ be a random vector such that the Mandelbrot cascades $m$ and $m'$ associated to the weight $W$ et $W'$ are non-degenerate. Suppose that :
\begin{itemize}
\item There exists $r>1$ such that $E[Y_\infty^r]<+\infty$ and $E[Y_\infty^{'r}]<+\infty$
\item There exists $\alpha>0$ such that $E[Y_\infty^{-\alpha}]<+\infty$
\end{itemize}
Then, almost surely,
$$\lim_{n\to +\infty}\frac{\log m(I_n(x))}{\log \vert I_n(x)\vert}=1-E[W'\log_\ell W]\quad dm'-\mbox{almost every where.}$$
\end{theorem}
\begin{proof}
 The ideas are quite similar to those developed in the proof of Theorem \ref{THEOdim}. The probability measure  on the product space $\tilde\Omega=\Omega\times[0,1]$ is now
$$Q'(A)=E\left[\int 1\!\!\!1_A\,dm'\right]$$
and the related expectation is denoted by $\E'$. The notations are the same as in Theorem \ref{THEOdim}. In particular
$$dm= f_n\,d\mu_n,\qquad f_n=g_1\times\cdots\times g_n,\qquad\frac{\log(m(I_n(x))}{\log\vert I_n(x)\vert}=\frac{\log f_n(x)+\log\mu_n(I_n(x))}{-n\log\ell}$$
and we have to prove :
\begin{enumerate}
\item $\frac1n\sum_{j=1}^n\log g_j$ converges to $E[W'\log W]$\quad$dQ'$ almost surely
\item $\frac1n\log\mu_n(I_n(x))$ converges to $-\log\ell$\quad$dQ'$ almost surely.
\end{enumerate}

Step 1 : {\it behavior of $\frac1n\sum_{j=1}^n\log g_j$.}

\noindent
In the same way as in Lemma \ref{LEMMA4}, we have :
$$
\E'[\phi(g_n)]
=\sum_{\eps\in\M n}E\left[\phi(W_{\eps})m'(I_{\eps})\right]
=E[\phi(W)W']
$$
when $\phi$ is a bounded measurable function and the $g_n$ are identically distributed. On the other hand,

\begin{eqnarray*}
\E'\left[\phi_1(g_1)\cdots\phi_n(g_n)\right]&=&
E[\phi_1(W)W']\times\cdots\times E[\phi_n(W)W']\\
&=&\E'\left[\phi_1(g_1)\right]\times\cdots\times\E'\left[\phi_n(g_n)\right]
\end{eqnarray*}
which proves the independance of the random variables $(g_n)$ with respect to $Q'$. Observing that $\E'[\vert \log g_n\vert]=E[W'\vert\log W\vert]<+\infty$, the strong law of large numbers says that
$$\frac1n\sum_{k=1}^n\log g_k\xrightarrow[n\to +\infty]{}E[W'\log W]\quad dQ'-\mbox{almost surely}.$$

Step 2 : {\it behavior of $\frac1n\log\mu_n(I_n(x))$.}

\noindent
Let $\eps=\e 1\cdots\e n\in\m n$ and recall that $m(I_\eps)=W_{\e 1}\cdots W_{\e 1\cdots\e n}\,\mu_n(I_\eps)$. It is easy to see that $\mu_n(I_\eps)$ is independent to $W_{\e 1},\cdots,W_{\e 1\cdots\e n}$ and has the same law as 
$\ell^{-n}Y_\infty$. If we write $m'(I_\eps)=W'_{\e 1}\cdots W'_{\e 1\cdots\e n}\,\mu'_n(I_\eps)$, we can more precisely say that the vector $(m(I_\eps),m'(I_\eps))$ and is identically distributed to $(\ell^{-n}Y_\infty,\ell^{-n}Y'_\infty)$ and independent to $W_{\e 1},\cdots,W_{\e 1\cdots\e n},W'_{\e 1},\cdots,W'_{\e 1\cdots\e n}$. It follows that
\begin{eqnarray*}
\E'\left[\left(\ell^n\mu_n(I_n(x))\right)^{-\eta}\right]&=&E\left[\ell^{-n\eta}\int\mu_n(I_n(x))^{-\eta}dm'(x)\right]\\
&=&\ell^{-n\eta}\sum_{\eps\in\M_n} E\left[\mu_n(I_\eps)^{-\eta}m'(I_\eps)\right]\\
&=&\ell^{-n\eta}\sum_{\eps\in\M_n} E\left[W'_{\e 1}\cdots W'_{\e 1\cdots\e n}\right]E\left[\mu_n(I_\eps)^{-\eta}\mu'_n(I_\eps)\right]\\
&=&E[Y_\infty^{-\eta}Y'_\infty]\\
&\le&E[Y_\infty^{-\eta r'}]^{1/r'}E\left[Y^{'r}_\infty\right]^{1/r}
\end{eqnarray*}
where $r'$ is such that $\frac1{r}+\frac1{r'}=1$. If we choose $\eta$ such that $\eta r'=\alpha$, we get 
$$\E'\left[\sum_{n=1}^{+\infty}\frac1{n^2}\left(\ell^n\mu_n(I_n(x))\right)^{-\eta}\right]<+\infty$$
and we can conclude as in Lemma \ref{LEMMA4} that
$$\mbox{almost surely,}\quad \liminf_{n\to +\infty}\frac{\log\left(\ell^n\mu_n(I_n(x))\right)}n\ge 0\quad dm'-\mbox{almost every where}.$$
In the same way,
\begin{eqnarray*}
\E'\left[\left(\ell^n\mu_n(I_n(x))\right)^{\eta}\right]&=&E[Y_\infty^{\eta}Y'_\infty]\\
&\le& E[Y_\infty^{\eta r'}]^{1/r'}E\left[Y^{'r}_\infty\right]^{1/r}
\end{eqnarray*}
which is finite and independent of $n$  if we choose  $\eta$ such that $\eta r'=r$. We can then conclude that
$$\mbox{almost surely,}\quad \limsup_{n\to +\infty}\frac{\log\left(\ell^n\mu_n(I_n(x))\right)}n\le 0\quad dm'-\mbox{almost every where}.$$
\end{proof}
\subsection{Application to the multifractal analysis of Mandelbrot cascades}
If we apply Theorem \ref{THEOsimultaneous} to the case where $W'=\frac{W^q}{E[W^q]}$, we obtain the following result on multifractal analysis of Mandelbrot cascades.
\begin{theorem}\label{THEOmultif}
Let $m$ be a Mandelbrot cascade associated to a weight $W$. Suppose that \em{(\ref{EQsimplification})} is satisfied and define $q_{min}$ and $q_{max}$ as above. Let $\beta=-\tau'(q)$ with $q\in(q_{min},q_{max})$. Then
$$\dim(E_\beta)\ge\tau^*(\beta)$$
where
$$E_\beta=\left\{x\ ;\ \lim_{n\to \infty}\frac{\log m(I_n(x))}{\log\vert I_n(x)\vert}=\beta\right\}.$$
\end{theorem}
\begin{proof}
 As suggested at the beginning of Section \ref{SECmulti}, let $W'=\frac{W^q}{E[W^q]}$. The condition $q\in(q_{min},q_{max})$ ensures that the associated cascade $m'$ is non-degenerate. More precisely, observing that $\tau'(1)<0$ and $\tau'_q(1)=-\tau^*(-\tau'(q))=-\tau^*(\beta)<0$, we can find $r>1$ such that $E[Y_\infty^r]<+\infty$ and $E[Y_\infty^{'r}]<+\infty$. Finally, all the hypotheses of Theorem \ref{THEOsimultaneous} are satisfied. Observe that 
$$1-E[W'\log_\ell W]=1-E\left[\frac{W^q}{E[W^q]}\log_\ell W\right]=-\tau'(q)=\beta.$$
The conclusion of Theorem \ref{THEOsimultaneous} says that almost surely, the set $E_\beta$ is of full measure $m'$. It follows that
$$
\dim(E_\beta)\ge\dim(m')
=-\tau'_q(1)=-q\tau'(q)+\tau(q)=\tau^*(-\tau'(q))=\tau^*(\beta).
$$
\end{proof}
\subsection{To go further}
It is natural to ask if the inequality proved in Theorem \ref{THEOmultif} is an equality. Indeed we know that the inequality 
$$\dim(E_\beta)\le \tilde\tau^*(\beta)\quad\mbox{where}\quad
\tilde\tau(q)=\limsup_{n\to+\infty}\frac1n\log_\ell\left(\sum_{I\in\f n}m(I)^q\right)$$
is always true (see for example \cite{BMP}). 

Our goal is then to compare the convex functions $\tau$ and $\tilde\tau$. Such a comparison can be deduced from the existence of negative moments for the random variable $Y_\infty$.
\begin{proposition}[Existence of negative moments]\label{PROPnegative}
Suppose that \em{(\ref{EQsimplification})} is satisfied. Then, for any $\alpha>0$,  $E\left[Y_\infty^{-\alpha}\right]<+\infty$.
\end{proposition}
\begin{proof}
The argument is developed for example in \cite{B99} or \cite{K91}. Let 
$$F(t)=E\left[e^{-tY_\infty}\right]$$ 
be the generating function of $Y_\infty$. The fundamental equation 
$ \ell Y_\infty=\sum_{j=0}^{\ell-1}W_jY_\infty(j)$ gives the following duplication formula :
\begin{equation}\label{EQdupli}
F(\ell t)=\left(\int_0^{+\infty}F(tw)\,dP_W(w)\right)^\ell.
\end{equation}
We claim that it is sufficient to prove that for any $\alpha>0$, $F(t)=O(t^{-\alpha})$ when $t\to+\infty$. Indeed, if it is the case, 
\begin{eqnarray*}
 P\left[Y_\infty\le t^{-1}\right]=P\left[e^{-tY_\infty}\ge e^{-1}\right]\le eF(t)=O(t^{-\alpha})
\end{eqnarray*}
and we can conclude that 
$$E\left[Y_\infty^{-\alpha'}\right]=\int_0^{+\infty}P\left[Y_\infty^{-\alpha'}\ge t\right]dt=\int_0^{+\infty}P\left[Y_\infty\le t^{-1/\alpha'}\right]dt<+\infty$$
for any $\alpha'<\alpha$.

Let us now observe that (\ref{EQdupli}) gives for any $t>0$ and any $u\in(0,1]$,
\begin{eqnarray*}
 F(t)&\le&\left(\int_0^{+\infty}F\left((t\ell^{-1})w\right)\,dP_W(w)\right)^2\\
&\le& \left(P[W\le \ell u] +F(tu)\right)^2\\
&\le& 2P[W\le \ell u]^2+2F(tu)^2.
\end{eqnarray*}
Moreover, for any $\beta>0$, assumption (\ref{EQsimplification}) ensures that
$$P\left[ W\le\ell u\right]=P\left[W^{-\beta}\ge(\ell u)^{-\beta}\right]\le(\ell u)^\beta E\left[W^{-\beta}\right]=Cu^\beta .$$
Proposition \ref{PROPnegative} is then a consequence of the following elementary lemma.
\begin{lemma}
 Let $\beta>0$ and $\psi\ :\ [0,+\infty)\to [0,+\infty)$ a continuous function such that 
$$\displaystyle\lim_{t\to +\infty}\psi(t)=0.$$
Suppose that there exists $K>0$ such that for any $t>0$ and any $u\in(0,1]$,
\begin{equation}\label{EQpsi}
\psi(t)\le Ku^{2\beta}+2\psi(tu)^2.
\end{equation}
Then, for any $\alpha<\beta$, $\psi(t)=O(t^{-\alpha})$ when $t\to+\infty$.
\end{lemma}
\begin{proof}
Let $\alpha<\beta$ and $t_0>1$ such that $4Kt_0^{2(\alpha-\beta)}+\frac12\le 1$. Let $\lambda >1$ such that 
$$\mbox{for any } t\in\left[t_0,t_0^2\right],\quad\psi(\lambda t)\le\frac1{4t^\alpha}.$$ 
Define $\psi_\lambda(t)=\psi(\lambda t)$. Equation (\ref{EQpsi}) remains true if we replace $\psi$ by $\psi_\lambda$. Moreover, if $u=\frac1t$ we get
$$\psi_\lambda(t^2)\le Kt^{-2\beta}+2\psi_\lambda(t)^2.$$ 
If $t\in\left[t_0,t_0^2\right]$, we obtain
\begin{eqnarray*}
 \psi_\lambda(t^2)&\le&K t^{-2\beta}+2\left(\frac1{4t^\alpha}\right)^2\\
&=&\frac1{4t^{2\alpha}}\left[4Kt^{2(\alpha-\beta)}+\frac12\right]\\
&\le&\frac1{4(t^2)^\alpha}.
\end{eqnarray*}
Define the sequence $(t_n)$ by $t_{n+1}=t_n^2$. Using the same argument, we obtain step by step
$$\mbox{for any }n\ge 0,\quad\mbox{for any }t\in [t_n,t_{n+1}],\qquad \psi_\lambda(t)\le\frac1{4t^\alpha}$$
and the conclusion follows.

\end{proof}
\end{proof}
\begin{corollary}\label{COROtau} Suppose that \em{(\ref{EQsimplification})} is satisfied. Then,
$$\mbox{almost surely,\quad  for any }q\in\R,\qquad\tilde\tau(q)\le\tau(q).$$
\end{corollary}

\begin{proof} 
Using the continuity of the convex functions $\tilde\tau$ ans $\tau$, it is sufficient to prove that for any $q\in\R$, almost surely, $\tilde\tau(q)\le\tau(q)$. Let
$$q_0=\sup\{ q>1\ ;\ \tau(q)<0\}.$$
It is possible that $q_0=+\infty$. Nevertheless, if $q_0<+\infty$ and if $q\ge q_0$, we obviously have $\tilde\tau(q)\le 0\le\tau(q)$. 

We can now suppose that $q<q_0$ and we claim that 
\begin{equation}\label{EQtruc}
E\left[Y_\infty^q\right]<+\infty.
\end{equation} 
Indeed, the case $q<0$ is due to Proposition \ref{PROPnegative}, the case $0\le q\le 1$ is obvious  and the case $1<q< q_0$ is due to Theorem \ref{THEOmoments}.

Let $\eps=\e 1\cdots\e n\in\m n$. As observed in Lemma \ref{LEMMA5}, we have 
$$m(I_\eps)=W_{\e 1}\cdots W_{\e 1\cdots\e n}\,\mu_n(I_\eps).$$
where $\mu_n(I_\eps)$ is independent to $W_{\eps_1},\cdots,W_{\e 1\cdots\e n}$ and has the same distribution as $\ell^{-n}Y_\infty$. It follows that
\begin{eqnarray*}
 E\left[\sum_{\eps\in\m n}m(I_\eps)^q\right]&=&\sum_{\eps\in\m n}E\left[W_{\e 1}^q\cdots W_{\e 1\cdots\e n}^q\,\mu_n(I_\eps)^q\right]\\
&=&\ell^nE\left[W^q\right]^n\ell^{-nq}E\left[Y_\infty^q\right]\\
&=&\ell^{n\tau(q)}E\left[Y_\infty^q\right].
\end{eqnarray*}
Let $t>\tau(q)$. In view of (\ref{EQtruc}) we get 
$$
E\left[\sum_{n\ge 1}\ell^{-nt}\sum_{\eps\in\m n}m(I_\eps)^q\right]=\sum_{n\ge 1}\ell^{-nt}\ell^{n\tau(q)}E\left[Y_\infty^q\right]<+\infty.
$$
It follows that almost surely, $\sum_{\eps\in\m n}m(I_\eps)^q\le \ell^{nt}$ if $n$ is large enough and we can conclude that $\tilde\tau(q)\le t$ almost surely. This gives the conclusion.
\end{proof}
We can now prove the following result.
\begin{theorem}\label{THEOegal}Suppose that \em{(\ref{EQsimplification})} is satisfied. Then, 
for any  $\beta\in(-\tau'(q_{max}),-\tau'(q_{min}))$,
$$\dim(E_\beta)=\tau^*(\beta)\quad\mbox{almost surely}.$$
\end{theorem}
Indeed Theorem \ref{THEOmultif} and  Corollary \ref{COROtau} ensure that for any $\beta\in(-\tau'(q_{max}),-\tau'(q_{min}))$,
$$\tau^*(\beta)\le\dim(E_\beta)\le\tilde\tau^*(\beta)\le\tau^*(\beta)$$
which gives the conclusion of Theorem \ref{THEOegal}.

\begin{remark}
 The proof of Theorem \ref{THEOegal} shows that 
$$\tau^*(\beta)=\tilde\tau^*(\beta)\quad \mbox{for any }\beta\in (-\tau'(q_{max}),-\tau'(q_{min})).$$ It follows that $\tau(q)=\tilde\tau(q)$ for any $q\in(q_{min},q_{max})$. When $q_{min}$ and $q_{max}$ are finite, it is possible to prove that 
$$\tilde\tau(q)=\tau'(q_{min})q\quad\mbox{if }q\le q_{min}\quad\mbox{ and }\quad\tilde\tau(q)=\tau'(q_{max})q\quad\mbox{if }q\ge q_{max}$$
 (see for example \cite{BFP}).
\end{remark}

Let us finish this text by recalling that Barral proved in \cite{B00} the much more difficult result:
$$\mbox{almost surely,\quad}
\left\{
\begin{aligned}
&\mbox{for any }\beta\in(-\tau'(q_{max}),-\tau'(q_{min})),\quad &\dim(E_\beta)&=\tau^*(\beta)\\
\null\\
&\mbox{for any }\beta\not\in[-\tau'(q_{max}),-\tau'(q_{min})],\quad  &E_\beta&=\emptyset .
\end{aligned}
\right.$$
\begin{acknowledgement}
I would like to thank St\'ephane Jaffard and St\'ephane Seuret who offered me the opportunity to  deliver this course during the GDR meeting at Porquerolles Island (September 22-27,  2013)
\end{acknowledgement}

\end{document}